\documentclass[a4paper,11pt]{article}
\usepackage[T1]{fontenc}
\usepackage{lmodern}
\usepackage{amsmath,amsthm,amssymb,amsfonts}
\usepackage[dvipsnames,svgnames,table]{xcolor}
\usepackage{pgf,tikz}
\usetikzlibrary{calc}
\usetikzlibrary{arrows.meta}
\usepackage{enumitem}
\usepackage{fullpage}
\usepackage{mathtools}
\usepackage{nicematrix}
\usepackage{booktabs}
\usepackage[title]{appendix}
\usepackage{thmtools, thm-restate}
\usepackage[colorlinks=true,linkcolor=RoyalBlue,urlcolor=RoyalBlue,citecolor=PineGreen]{hyperref}
\usepackage[nameinlink]{cleveref}

\newtheorem{theorem}{Theorem}[section]
\newtheorem{proposition}[theorem]{Proposition}
\newtheorem{lemma}[theorem]{Lemma}
\newtheorem{corollary}[theorem]{Corollary}
\newtheorem{conjecture}[theorem]{Conjecture}
\newtheorem{question}[theorem]{Question}
\newtheorem{claim}{Claim}
\newtheorem{problem}[theorem]{Problem}

\theoremstyle{definition}
\newtheorem{definition}[theorem]{Definition}
\newtheorem{example}[theorem]{Example}

\theoremstyle{remark}
\newtheorem{remark}[theorem]{Remark}

\newenvironment{claimproof}[1][\proofname]
	{
		\proof[#1]%
			
	}
	{
		\endproof
	}

\newcommand{\NN}{\mathbb{N}}
\newcommand{\CC}{\mathbb{C}}
\newcommand{\RR}{\mathbb{R}}

\newcommand{\ZZ}{\mathbb{Z}}
\newcommand{\FF}{\mathbb{F}}
\newcommand{\PP}{\mathbb P}
\newcommand{\suchthat}{\;\ifnum\currentgrouptype=16 \middle\fi|\;}
\newcommand{\bigmid}{\left.\vphantom{\Big\{} \suchthat \vphantom{\Big\}}\right.}

\newcommand{\lin}{{\mathrm{lin}}}

\newcommand{\cB}{\mathcal{B}}

\newcommand{\cF}{\mathcal{F}}
\newcommand{\cG}{\mathcal{G}}
\newcommand{\bo}{\mathbf{1}}

\newcommand{\countsymb}{\mathbf{c}}
\newcommand{\ccountd}[2]{\countsymb_{#1}(#2)}
\newcommand{\ccount}[1]{\ccountd{2}{#1}}
\newcommand{\ccountsym}[2]{\countsymb_2^{\circlearrowleft #2}(#1)}
\newcommand{\ccountper}[2]{\countsymb_2^{\rightleftarrows #2}(#1)}

\newcommand{\ee}{\mathrm{e}}
\newcommand{\ii}{\mathfrak{i}}

\DeclareMathOperator{\codim}{codim}
\DeclareMathOperator{\conv}{conv}
\DeclareMathOperator{\mult}{mult}

\DeclareMathOperator{\Span}{Span}
\DeclareMathOperator{\Trop}{Trop}
\DeclareMathOperator{\rank}{rank}
\DeclareMathOperator{\nbc}{nbc}
\DeclareMathOperator{\im}{im}
\DeclareMathOperator{\Tropp}{\overline{Trop}}
\DeclareMathOperator{\aff}{aff}
\DeclareMathOperator{\rec}{rec}

\DeclareMathOperator{\Star}{star}

\colorlet{colbg}{white}
\colorlet{colfg}{black}
\colorlet{colgraphv}{colfg!75!white}
\colorlet{colgraphe}{colfg!55!white}
\colorlet{colG}{DarkSeaGreen}
\definecolor{colR}{HTML}{CC6677}
\definecolor{colO}{HTML}{DDCC77}
\definecolor{colB}{HTML}{6699CC}
\colorlet{colY}{Gold!90!black}

\tikzstyle{vertex}=[fill=colgraphv,circle,inner sep=0pt, minimum size=4pt]
\tikzstyle{edge}=[line width=1.5pt,colgraphe]
\tikzstyle{dedge}=[edge, -{LaTeX[round,length=8pt]}]

\tikzstyle{labelsty}=[font=\scriptsize]

\title{Counting fibres of the Hadamard product using Bergman fans}

\author{Oliver Clarke\thanks{Department of Mathematical Sciences, Durham University}
\and
Sean Dewar\thanks{Numerical Analysis and Applied Mathematics (NUMA), KU Leuven}
\and
Matteo Gallet\thanks{Department of Mathematics, Informatics and Geosciences, University of Trieste}
\and
Georg Grasegger\thanks{Johann Radon Institute for Computational and Applied Mathematics, Austrian Academy of Sciences}
\and
Daniel Green Tripp\thanks{School of Mathematics, University of Bristol}
\and
Ben Smith\thanks{School of Mathematical Sciences, Lancaster University}
}
\date{}

\begin{document}

\maketitle

\begin{abstract}
    We study the generic fibre of the Hadamard product of linear spaces via matroid theory and tropical geometry.
    To do so, we introduce the \emph{flip product}, a numerical invariant associated to a pair of matroids defined via the stable intersection of their (flipped) Bergman fans.
    Our first main result is that the cardinality of a generic fibre for the Hadamard product of linear spaces is exactly the flip product of their matroids.
    We also provide a recursive algorithm for computing the flip product of any pair of matroids.
    As an application of our techniques, we extend the notion of realisation numbers from rigidity theory to rotational-symmetric and periodic realisation numbers and we provide combinatorial algorithms to compute them.
    Finally, we show a number of existing matroid invariants are specialisations of the flip product, including the beta invariant.
\end{abstract}

\section{Introduction}\label{sec:intro}

Given a finite indexing set $E$ and linear spaces $U, V \subset \mathbb{C}^E$,
we define the \emph{Hadamard map} of~$U,V$ to be the bilinear map
\begin{equation*}
    f_{U,V} \colon U \times V \rightarrow \mathbb{C}^E, ~ (u,v) \mapsto u \odot v := (u_e v_e)_{e \in E}.
\end{equation*}
The Zariski closure of $f_{U,V}(U \times V)$ is called the \emph{Hadamard product} of $U$ and $V$, and is denoted by~$U \odot V$.
The Hadamard map is invariant under the natural action of the algebraic torus~$\mathbb{C} \setminus \{0\}$ on $U \times V$ that sends a pair $(u,v) \in U \times V$ and a scalar $t \in \mathbb{C} \setminus \{0\}$ to the pair $(tu, t^{-1}v)$ since $(t u ) \odot (t^{-1} v) = u \odot v$.
With this in mind, our aim in this paper is to address the following problem:

\begin{problem}\label{problem}
    Determine the number of points contained in a generic fibre of the Hadamard map~$f_{U,V}$ modulo the natural torus action on $U \times V$.
\end{problem}

The first appearance of Hadamard products of algebraic varieties can be traced to the papers~\cite{CuetoMortonSturmfels2010, CuetoTobisYu2010}.
After that, several papers and books appeared on the topic, including \cite{BocciCarlini2024,BocciCarliniKileel2016}.
The Hadamard product of two linear spaces $U, V \subset \mathbb{C}^E$ is closely tied to the behaviour of their \emph{algebraic matroids} $M(U)$ and $M(V)$; see \cite{RosenSidmanTheran2020}.
For example, Bernstein \cite{Bernstein2022} proves that the dimension of the Hadamard product $U \odot V$ is entirely determined by the matroids~$M(U)$ and~$M(V)$; see \Cref{thm:bern_matroid_linear}.

\subsection{Introducing: the flip product of matroids}

In this paper, we introduce the following concept that was previously alluded to in \cite{ClarkeDewarEtAl2025}. An early version of the concept can be found in \cite[Proposition~5.2 and Lemma~6.1]{HuhKatz2012} for the special case where one of the matroids is uniform.

\begin{definition}
\label{definition:flip_product}
The \emph{flip product} of matroids~$M$ and~$N$ with shared ground set~$E$ is given by the formula
\begin{equation*}
    M \ast N \ := \ \bigl(-\Tropp(N)\bigr) \cdot \Tropp(M) \ \in \ \mathbb{Z}_{\geq 0} \cup \{\infty\},
\end{equation*}
where $\Tropp (M)$ is the projective Bergman fan of $M$ (see \Cref{def:projectiveBergmanFan}),
$-\Tropp(N)$ is the tropical cycle $\{-x : x \in \Tropp(N)\}$,
and the operation ``$\,\cdot\,$'' is the tropical intersection product (see \Cref{subsec:tropical_intersection_product}).
\end{definition}

In addition to the convention that $M \ast N = 1$ when $E= \emptyset$,
the flip product satisfies the following properties:
\begin{enumerate}
    \item $M \ast N = N \ast M$;
    \item if either $M$ or $N$ contain a loop, then $M \ast N = 0$ (\Cref{prop:loopbad});
    \item if $r(M) + r(N) < |E|+1$, then $M \ast N = 0$ and if $r(M) + r(N) > |E|+1$, then $M \ast N \in \{0,\infty\}$ (\Cref{prop:intersection+number}); here $r(\cdot)$ denotes the rank of a matroid.
\end{enumerate}

\subsection{Main results}

Our first main result describes how flip products can be used to solve \Cref{problem}.
Specifically, we show that for any $\mathbb{C}$-representable matroids $M,N$ with the same ground set $E$,
the flip product $M \ast N$ can be geometrically interpreted as the degree of the Hadamard product of the $\mathbb{C}$-representations for $M$ and $N$.

\begin{restatable}{theorem}{mainalg}\label{thm:mainalg}
    Let $U, V \subset \mathbb{C}^E$ be linear subspaces and let $M(U)$ and $M(V)$ be the representable matroids defined by~$U$ and~$V$, respectively.
    Then, for any generic $\lambda \in \mathbb{C}^E$ we have
    \begin{equation*}
        \#\bigl(f_{U,V}^{-1}(\lambda)/\!\!\sim\bigr) = M(U) \ast M(V),
    \end{equation*}
    where $\sim$ is the equivalence relation on $U\times V$ where $(u,v) \sim (u',v')$ if and only if
    $u'=tu$ and $v'=t^{-1}v$ for some $t \in \CC\setminus \{0\}$, i.e., the pairs $(u,v)$ and $(u',v')$ are in the same orbit of the natural torus action on $U \times V$.
\end{restatable}

Notice that the condition $\dim(U) + \dim(V) = |E| + 1$ is necessary in order to ensure a finite number of preimages modulo torus action.

Our second main result provides a recursive algorithm for computing flip products of matroids when combined with the following base case:
if $|E|=1$, then
\begin{equation*}
    M \ast N =
    \begin{cases}
        1 &\text{if } M\cong U_{1,1} ~ \text{and} ~ N\cong U_{1,1}, \\
        0 &\text{otherwise},
    \end{cases}
\end{equation*}
where $U_{1,1}$ is the rank-1 uniform matroid with 1 element (\Cref{lem:1element}).

\begin{restatable}{theorem}{main}\label{thm:main}
    Let $M, N$ be loopless matroids on $E$ such that $r(M) + r(N) = |E| + 1$,
    and let $\epsilon \in E$.
    Then the following equality holds:
    \begin{equation*}
        M \ast N = \sum_{E_1, E_2} \Big( (M / E_1) \ast (N \setminus E_1) \Big) \Big( (M \setminus E_2) \ast (N /E_2) \Big),
    \end{equation*}
    where each pair $E_1, E_2 \subset E$ satisfies $E_1 \cup E_2 = E$ and
    $E_1 \cap E_2 = \{ \epsilon \}$.
\end{restatable}

Using the convention that $0 \cdot \infty = 0$,
the summation given in \Cref{thm:main} always provides a finite value;
see \Cref{rem:notinfinity} for more details.

\Cref{thm:main} is the matroidal analogue of the main result of Capco et al.\ \cite[Theorem~4.7]{CapcoGalletEtAl2018}.
There, they combine techniques from algebraic geometry and ideas from tropical geometry to obtain their result, which only applies to graphic matroids.
In contrast, our proof solely uses tropical geometry because neither matroid in the flip product is required to be realisable.

\begin{remark}\label{rem:characteristic}
    The authors believe that \Cref{thm:mainalg} should extend to Hadamard products of linear spaces over other algebraically closed fields with odd or zero characteristic,
    although proving this would require generalising many of the tools used throughout the paper (e.g., Bernstein's theorem --- \Cref{thm:bern_matroid_linear}).
    If this generalisation does indeed go through with no issues, then \Cref{thm:main} can be immediately applied to provide an inductive algorithm for computing the cardinality of generic fibres for the relevant Hadamard map.
    It is, however, unclear to the authors whether \Cref{thm:mainalg} holds for fields of characteristic 2.
    For example, many of the key steps in proving \Cref{thm:mainalg} fail over characteristic 2 since they require using squared values.
\end{remark}

\subsection{Applications: counting symmetric/periodic framework realisations}

One important application of \Cref{thm:mainalg} and \Cref{thm:main} is within the field of rigidity theory.
Here, we consider a graph to yield a bar-and-joint framework in the plane,
with edges modelling stiff bars and vertices modelling flexible joints. 
Formally, given a graph~$G = ([n], E)$ and a natural number~$d$, we define a realisation of~$G$ to be a map $p \colon [n] \rightarrow \RR^d$ and we say that the pair $(G,p)$ is a \emph{framework}.
If the only continuous motions of the framework that preserve edge lengths are rigid body motions, then the framework is said to be \emph{rigid}.
Asimow and Roth \cite{AsimowRoth1978} proved that rigidity is a generic property, in that either every generic framework is rigid, or none of them are.
Hence, we say that a graph is \emph{rigid} if any generic framework $(G,p)$ is rigid.

An equivalent definition for a generic framework $(G,p)$ to be rigid is that the number of frameworks $(G,q)$, modulo isometries, with the same edge lengths is finite.
Unlike rigidity, this value is not a generic property if we restrict ourselves to real frameworks.
However, if we allow for complex solutions, this value (here denoted by $\ccountd{d}{G}$) is the same for all generic frameworks.
It was proven by Clarke et al.~\cite{ClarkeDewarEtAl2025} that when the graph in question is \emph{minimally rigid} --- in that the graph contains no proper spanning rigid subgraph --- then the \emph{realisation number} $\ccount{G}$ is exactly half the flip product of two copies of the graphical matroid for $G$:

\begin{theorem}[Clarke et al.~\cite{ClarkeDewarEtAl2025}]\label{thm:originalrigid}
    Let $G$ be a minimally rigid graph.
    Then, given $M(G)$ is the graphical matroid for $G$, we have
    \begin{equation*}
        \ccount{G} = \frac{1}{2} M(G) \ast M(G).
    \end{equation*}
\end{theorem}

Building on this idea, the paper \cite{BielbyChauhanEtAl2025} studies how the product $\bigl(-\Trop(\Sigma_1)\bigr) \cdot \Trop(\Sigma_2)$ behaves, where $\Sigma_1$ and $\Sigma_2$ are tropical stars along rays of $\Trop\bigl(M(G)\bigr)$, with the aim to provide better lower bounds for the realisation number~$\ccount{G}$.

Using \Cref{thm:main}, we extend the ideas in \cite{ClarkeDewarEtAl2025} to counting the number of edge-length equivalent frameworks when some type of symmetry is enforced.
The two types of symmetric rigidity that we can approach using Hadamard products are \emph{$k$-fold rotation symmetric rigidity}, where every framework is assumed to have $k$-fold rotational symmetry, and \emph{periodic rigidity}, where every framework is assumed to be an infinite framework that has either 1- or 2-periodic symmetry.
Since we only care about our frameworks with enforced symmetry, we can quotient out any graph by a particular symmetry $\Gamma$ to form a \emph{$\Gamma$-gain graph} (see \Cref{subsec:gain}).

The number of complex symmetric edge-length equivalent frameworks modulo isometries is again a generic (up to enforced symmetry) property.
By adapting \Cref{thm:mainalg},
we obtain the following flip product expressions for the symmetric variants of $\ccount{G}$.
All relevant terminology can be found in \Cref{applications}.

\begin{restatable}{theorem}{rotation}\label{thm:rotation}
    Let $(G,\phi)$ be a $\ZZ_k$-gain graph.
    If $(G,\phi)$ is minimally $k$-fold rotation symmetrically rigid,
    then, given $M(G,\phi)$ is the $\ZZ_k$-gain graphical matroid for $(G,\phi)$, we have
    \begin{equation*}
        \ccountsym{G,\phi}{k} = \frac{1}{2} M(G,\phi) \ast M(G,\phi).
    \end{equation*}
\end{restatable}

\begin{restatable}{theorem}{periodic}\label{thm:periodic}
    Let $(G,\phi)$ be a $\mathbb{Z}^k$-gain graph for $k \in \{1,2\}$.
    If $(G,\phi)$ is minimally periodically rigid,
    then, given $M(G,\phi)$ is the $\ZZ^k$-gain graphical matroid for $(G,\phi)$, we have
    \begin{equation*}
        \ccountper{G,\phi}{k} = \frac{1}{2} M(G,\phi) \ast M(G,\phi).
    \end{equation*}
\end{restatable}

Both \Cref{thm:rotation} and \Cref{thm:periodic} can be combined with \Cref{thm:main} to provide an inductive formula for computing symmetric realisation numbers.
In fact, this can be streamlined into an inductive formula using gain graphs instead of matroids (see \Cref{subsubsec:rotationformula,subsubsec:periodicformula}).
The immediate advantage here is that matroids can require a lot of information for storage (such as a list of bases) while gain graphs require relatively little.

\subsection{In the language of Chow rings and toric varieties}

Flip products admit another description when we think of balanced fans as classes in the Chow ring of some complete toric variety.
We very briefly sum up this construction, leaving to the interested reader to go into more depth via the references provided afterwards.

Let $\Delta$ be a rational fan in $\RR^n$ with $|\Delta|=\RR^n$ and let $X(\Delta)$ the corresponding toric variety.
A \emph{codimension $k$ Minkowski weight} on $\Delta$ is the data of multiplicities $\mult(\sigma) \in \ZZ$ for each cone $\sigma\in\Delta$ of dimension $n-k$ which makes the $(n-k)$-skeleton of $\Delta$ balanced.
With addition given by adding the weights on each cone, the set $\mathrm{MW}^k(\Delta)$ of codimension $k$ Minkowski weights on~$\Delta$ is an abelian group.
Furthermore,
\begin{equation*}
    \mathrm{MW}^\bullet(\Delta) := \bigoplus_{k=0}^n\mathrm{MW}^k(\Delta)
\end{equation*}
is a graded ring with product given by stable intersection (see \Cref{def:stableIntersection}). Fulton and Sturmfels proved in \cite{FultonSturmfels1997} that $\mathrm{MW}^\bullet(\Delta)$ and the Chow cohomology ring $A^\bullet\bigl(X(\Delta)\bigr)$ (as defined by Fulton and MacPherson, see~\cite{Fulton1998,FultonMacPherson1981}) are canonically isomorphic as graded rings.

A particular case of interest for us is when $\Delta=\Delta_n$ is the fan underlying $\Tropp(U_{[n],n})$ (see \Cref{def:BergmanFan}), where $U_{[n],n}$ is the rank $n$ uniform matroid with ground set $[n]$. Notice that the (projective) Bergman fan of any matroid $M$ of rank $r$ with ground set $[n]$ defines a codimension $n-r$ Minkowski weight on $\Delta_n$ and thus defines a class $\alpha_M\in A^{n-r}\bigl( X(\Delta_n) \bigr)$. Finally, the linear map $\Delta_n\to\Delta_n$ given by scalar multiplication by $-1$ defines a toric automorphism $\mathrm{Crem} \colon X(\Delta_n)\rightarrow X(\Delta_n)$.
Each pullback $\mathrm{Crem}^\ast (\alpha)\in A^\bullet\bigl(X(\Delta_n)\bigr)$ is the class corresponding to the Minkowski weight $\mult_{\mathrm{Crem}^\ast (\alpha)}(\sigma):=\mult_{\alpha}(-\sigma)$; that is, if we consider $\alpha$ as a balanced fan, then $\mathrm{Crem}^\ast (\alpha)=-\alpha$. In this language, we can restate \Cref{definition:flip_product} as the following:

\begin{definition}\label{def:chow}
    The \emph{flip product} of matroids $M$ and $N$ with shared ground set $[n]$ is given by the following formula, calculated in $A^\bullet(X(\Delta_n))$:
    \begin{equation*}
        M\ast N = \deg \bigl( \alpha_M \smallsmile \mathrm{Crem}^\ast (\alpha_N) \bigr) \,,
    \end{equation*}
    where $\smallsmile$ denotes the cup product in $A^\bullet\bigl(X(\Delta_n)\bigr)$.
\end{definition}

Using \Cref{def:chow}, any results presented in this work concerning flip products can be restated as intersection products in some toric variety. We use this interpretation in \Cref{matroid_application:characteristic_polynomial}.
For further information about this approach, see \cite{AdiprasitoHuhKatz2018,Ardila2024,HuhKatz2012}.

\subsection{Structure of paper}
The structure of the paper is as follows.
In \Cref{preliminaries} we provide background on the most important concepts that are used throughout the paper, such as matroids, the tropical intersection product, Bergman fans, and Hadamard products of matroids.
Additional background material for tropical geometry can also be found in \Cref{tropical_background}.
In \Cref{flip_products} we introduce the flip product and investigate some of its basic properties, mostly regarding when it attains values~$0$ and~$\infty$.
We prove \Cref{thm:mainalg} and \Cref{thm:main} in \Cref{hadamard_products} and \Cref{algorithm} respectively.
In \Cref{applications} we prove both \Cref{thm:rotation} and \Cref{thm:periodic} on using \Cref{thm:mainalg}.
After doing so, we then use \Cref{thm:main} to provide an iterative combinatorial algorithm for computing realisation numbers for symmetric frameworks.
In \Cref{matroid_applications} we provide characterisations for beta invariants and characteristic polynomials using the flip product.
In \Cref{weak_order},
we investigate the behaviour of the flip product with respect to matroid partitions.
We conclude the paper in \Cref{computational} with a computational analysis of flip products for matroids with small ground sets.

\section{Preliminaries}
\label{preliminaries}

In this section, we cover the preliminaries for algebraic geometry and tropical geometry needed throughout the paper.
First, we remark on the notion of `genericity' used in this paper.

Given an irreducible algebraic set $S$ over a field $K$, we say that a property $P$ of the points in $S$ holds for \emph{almost all} points of $S$ or \emph{generic} points of $S$, if $P$ holds for all points in some non-empty Zariski-open subset of $S$.
For $K=\RR$ or $K=\CC$ specifically, the definition of `generic' used here is a strictly weaker notion than what is usually used in most rigidity theory literature. There, a point $(x_i)_{i \in [n]} \in \CC^n$ is `generic' if $x_1,\dots,x_n$ are algebraically independent over $\mathbb{Q}$.

\subsection{Matroid notation and terminology}

We here introduce a few pieces of notation and terminology that we are freely using in this paper.

Given a matroid~$M$, we denote by $r(M)$ its \emph{rank}, and more generally we denote by $r_M(\cdot)$ its \emph{rank function}.
We denote the \emph{dual} of $M$ (i.e., the matroid whose bases are the complements of bases of $M$) by $M^*$.
An element $e$ of $M$ is a \emph{loop} if the set $\{e\}$ is a circuit.
An element $f$ of $M$ is a \emph{coloop} (also known as a \emph{bridge}) if $f$ is a loop of $M^*$;
equivalently, $f$ is a coloop if and only if $f$ is not contained in any circuit.

There are two important families of matroids we need throughout the paper:
\begin{itemize}
    \item Given a finite set~$E$ and a number $r \in \NN$ with $r \leq |E|$, we denote by~$U_{E,r}$ the \emph{uniform matroid} of rank $r$ on the ground set~$E$, i.e., the matroid whose bases are all and only the subsets of $E$ of cardinality~$r$. When $E = [n]$, we write $U_{n,r}$ instead of~$U_{[n],r}$.
    \item The \emph{graphic matroid} $M_G$ for a (multi)graph\footnote{All multigraphs are assumed to have finitely many vertices and edges. We reserve the term `graph' for multigraphs with no parallel edges or loops.} $G$ with edge set $E$ is the matroid with ground set~$E$ and circuits given by the edge subsets that form cycles of~$G$.
\end{itemize}

Given a matroid $M$ and a subset $F$ of the ground set $E$,
we denote the \emph{contraction of $F$ in $M$} by $M/F$ and the \emph{deletion of $F$ from $M$} by $M \setminus F$.
Throughout the paper we use, without proof, the elementary fact that $(M/F)^*=M^* \setminus F$ and $(M\setminus F)^*=M^* / F$.

Given two matroids $M_1$ and $M_2$ on two disjoint ground sets $E_1$ and $E_2$, we define the \emph{direct sum} of~$M_1$ and~$M_2$, denoted $M_1 \oplus M_2$ as the matroid whose ground set is $E_1 \sqcup E_2$ and whose collection of independent sets is the union of the collections of independent sets of~$M_1$ and~$M_2$.

A key aspect in the paper is the assignment of matroids to algebraic structures via their coordinate projections.
We recall that a morphism of varieties $f \colon X \rightarrow Y$ is \emph{dominant} if the image of $f$ is Zariski-dense in $Y$, i.e., $\overline{f(X)} = Y$.
Notice that here by ``variety'' we mean an irreducible algebraic set.

\begin{definition}
    Given a variety $X \subset \mathbb{C}^E$,
    we construct the \emph{algebraic matroid} $M(X)$ with ground set $E$ and independent sets given by
    \begin{equation*}
        F \text{ independent in } M(X) \quad \iff \quad \pi_F \colon X  \rightarrow \mathbb{C}^E, ~ (x_e)_{e \in E} \mapsto (x_e)_{e \in F} \text{ is dominant}.
    \end{equation*}
\end{definition}

\begin{example}
    Our key example of an algebraic matroid is $M(U)$, where $U \subseteq \CC^E$ is a linear space.
    In this case, the algebraic matroid has a particularly simple form.
    Let $A$ be any $|E| \times \dim(U)$ matrix whose column span equals $U$.
    Then $M(U)$ is just the row matroid on $A$.
    For more general varieties $X \subseteq \CC^E$, their algebraic matroids $M(X)$ are also representable over $\CC$, but the construction is more involved; see~\cite{RosenSidmanTheran2025}.
\end{example}

\subsection{Tropical intersection product}\label{subsec:tropical_intersection_product}

We freely use terminology from tropical geometry.
For the user's convenience, we recall the most basic concepts in \Cref{tropical_background}.

We denote the stable intersection of balanced polyhedral complexes $\Sigma_1, \ldots, \Sigma_k$ contained in $\mathbb{R}^n$ by $\Sigma_1 \wedge \dotsb \wedge \Sigma_k$.
We define the \emph{tropical intersection product} $\Sigma_1  \cdot \Sigma_2 \cdots \Sigma_k$ to be the cardinality of $\Sigma_1 \wedge \dotsb \wedge \Sigma_k$ counted with multiplicity.
We use the conventions that when $X \wedge Y$ is empty, $X \cdot Y = 0$, and when $X \wedge Y$ is of positive dimension, we say $X \cdot Y = \infty$.
The following result describes when the stable intersection of two polyhedral complexes is empty (and hence their tropical intersection product is zero).

\begin{theorem}[see {\cite[Theorem~3.6.10]{MaclaganSturmfels2015}}] \label{thm:stable+intersection}
    Let $\Sigma_1$ and $\Sigma_2$ be balanced polyhedral complexes in~$\RR^n$ of codimensions~$d$ and~$e$, respectively, and let $\ell$ be the dimension of the intersection of their lineality spaces.
    Then the stable intersection $\Sigma_1 \wedge \Sigma_2$ is either empty or a balanced polyhedral complex of codimension $d+e$.
    If $d + e + \ell > n$, then the stable intersection is empty.
\end{theorem}

We can give a precise characterisation of when the stable intersection is non-empty.
The proof of the following lemma is due to Yue Ren.

\begin{lemma}
    \label{lem:StableIntersectionNonempty}
    Let $\Sigma_1$ and $\Sigma_2$ be as in \Cref{thm:stable+intersection} and with only positive multiplicities.
    Then
    \begin{align*}
        \Sigma_1\wedge\Sigma_2\neq\emptyset \quad&\Longleftrightarrow \quad \exists \sigma_1\in\Sigma_1, \exists \sigma_2\in\Sigma_2: \aff(\sigma_1)+\aff(\sigma_2)=\RR^n \\
        \quad&\Longleftrightarrow \quad \dim(\Sigma_1 + \Sigma_2) = n
    \end{align*}
\end{lemma}
\begin{proof}
    We begin with the first equivalence.

    \begin{description}
        \item[($\Rightarrow$)]
        If $\aff(\sigma_1)+\aff(\sigma_2)\neq\RR^n$ for all $\sigma_1\in\Sigma_1$ and all $\sigma_2\in\Sigma_2$, then $\dim(\sigma_1+\sigma_2)\neq n$ for all $\sigma_1\in\Sigma_1$ and all $\sigma_2\in\Sigma_2$.
        Hence, $\Sigma_1\wedge\Sigma_2=\emptyset$ by definition.
        \item[($\Leftarrow$)]
        We may assume without loss of generality that the $\sigma_i$ are maximal, so that $\Star_{w_i}(\sigma_i)=\aff(\sigma_i)-w_i$ for $w_i\in\mathrm{relint}(\sigma_i)$. Then, as $\aff(\sigma_1) + \aff(\sigma_2) = \RR^n$, we have that $\Star_{w_1}(\sigma_1)+\Star_{w_2}(\sigma_2)=\RR^n$. Hence, by definition of stable intersection and \Cref{thm:stable+intersection}, $\Star_{w_1}(\sigma_1) \wedge \Star_{w_2}(\sigma_2) = \Star_{w_1}(\sigma_1) \cap \Star_{w_2}(\sigma_2)$ is of codimension $d+e$.
        Let $L$ be the $d+e$ dimensional linear space in $\RR^n$ orthogonal to $\Star_{w_1}(\sigma_1) \cap \Star_{w_2}(\sigma_2)$, so that they have a unique point of intersection.
        By \cite[Lemma 5.6]{ClarkeDewarEtAl2025}, we have
        \begin{equation*}
            \Sigma_1\cdot\Sigma_2\cdot L \geq \Star_{w_1}(\Sigma_1)\cdot\Star_{w_2}(\Sigma_2)\cdot L > 0.
        \end{equation*}
        Hence, $\Sigma_1\wedge\Sigma_2\neq\emptyset$.
    \end{description}
    The second equivalence follows from the fact that $\Sigma_1 + \Sigma_2$ is the finite union of cells $\sigma_1 + \sigma_2$, hence
    \begin{equation*}
        \dim(\Sigma_1 + \Sigma_2) = \max_{\sigma_i \in \Sigma_i}\bigl(\dim(\sigma_1 + \sigma_2)\bigr) \,. \qedhere
    \end{equation*}
\end{proof}

\subsection{Bergman fans}

We now introduce one of the key players in the paper: the Bergman fan.

\begin{definition}\label{def:BergmanFan}
    Let $M$ be a matroid on $E$. The \emph{Bergman fan} of $M$, denoted $\Trop(M)$, is the polyhedral fan contained in $\mathbb{R}^E$ given by
    \begin{equation*}
        \left\{
          (w_e)_{e \in E}
          \, : \,
          \min_{e \in C} w_e
          \text{ attains its minimum at least twice for every circuit }
          C \text{ of } M
        \right\}
    \end{equation*}
    with the fan structure induced as follows.
    For any point $x \in \mathbb{R}^E$ and any $t \in \mathbb{R}$,
    we fix
    \begin{equation*}
        E_{x\geq t} := \{ e \in E : x_e \geq t\}.
    \end{equation*}
    Then $w,w'\in \Trop(M)$ are in the relative interior of the same cone if
    \begin{equation*}
        \{E_{w\geq t} : t\in \RR\}=\{E_{w'\geq t} : t\in \RR\} \,.
    \end{equation*}
\end{definition}

\begin{remark}
    The fan structure defined above is referred to as the \emph{fine structure} of the Bergman fan.
    This was first introduced in \cite{ArdilaKlivans2006}.
\end{remark}

We note that if $M$ has a loop then $\Trop(M)$ is the empty set, as the set containing such loop is a circuit for which no point can attain its minimum at least twice.
On the other hand, if $M$ is loopless then $(0,\ldots,0)\in\Trop(M)$. In this case, if $w\in \Trop(M)$, then each set $E_{w\geq t}$ is a flat of $M$ and $\{E_{w\geq t} : t\in \RR\}$ is totally ordered by inclusion, defining a chain of flats which contains~$\emptyset$ and~$E$.
Moreover, any such chain of flats $\mathcal{F}:\emptyset=F_0\subsetneq F_1\subsetneq\cdots\subsetneq F_k=E$ defines a cone of $\Trop(M)$
\begin{equation*}
    \sigma_\mathcal{F}=\RR_{\geq0}\cdot\chi_{F_1}+\cdots+\RR_{\geq0}\cdot\chi_{F_{k-1}}+\RR\cdot\chi_{E}
\end{equation*}
where $\chi_F\in\RR^E$ denotes the indicator function of a subset $F\subseteq E$ (i.e., $\chi_F(e) = 1$ if $e \in F$, $\chi_F(e) = 0$ if $e \notin F$).

Note that, in particular, each cone $\sigma_\mathcal{F}$ contains $\RR\cdot\chi_{E}=\RR\cdot\bo$ in its lineality space, so that the quotient $\sigma_\mathcal{F}/\mathbb{R} \cdot \bo$, where $\bo$ is the all-ones vector, is a well defined cone in $\mathbb{R}^E/\mathbb{R} \cdot \bo$.

\begin{definition}\label{def:projectiveBergmanFan}
    Let $M$ be a matroid on $E$.
    We define \emph{projective Bergman fan}, denoted $\Tropp(M)$, to be the image of the Bergman fan of~$M$ in the quotient~$\mathbb{R}^E/ \mathbb{R} \cdot \bo$.
\end{definition}

Both the Bergman fan and projective Bergman fan are balanced fans where every maximal cell is assigned multiplicity~$1$ (see \Cref{tropical_background:bcp}).

\subsection{Generic root counts}

So as to better consider the properties of generic fibres, we introduce the following concepts.

\begin{definition}
    Let
    \begin{equation*}
      \CC[a][x^\pm]\coloneqq \CC\Bigl[a_{j}\mid j\in [m] \Bigr]\Bigl[x_i^{\pm1}\mid i\in [n]\Bigr]
    \end{equation*}
    be a \emph{parametrised (Laurent) polynomial ring} with parameters $a_j$ and variables $x_i$. Let $f\in \CC[a][x^\pm]$ be a parametrised polynomial, say $f=\sum_{\alpha\in\ZZ^n}c_\alpha x^\alpha$ with $c_\alpha\in \CC[a]$, and let $I\subseteq \CC[a][x^\pm]$ be a parametrised polynomial ideal.  We define their \emph{specialisation} at a choice of parameters $P\in \CC^m$ to be
    \begin{equation*}
        f_P\coloneqq\sum_{\alpha\in\ZZ^n}c_\alpha(P)\cdot x^\alpha\in \CC[x^\pm] \quad\text{and}\quad I_P\coloneqq\langle h_P\mid h \in I \rangle\subseteq \CC[x^\pm].
    \end{equation*}
    The \emph{root count} of $I$ at $P$ is defined to be the vector space dimension
    \begin{equation*}
        \ell_{I,P}\coloneqq \dim_\CC (\CC[x^\pm]/I_P)\in\ZZ_{\geq 0}\cup\{\infty\}.
    \end{equation*}
\end{definition}

\begin{definition}
    Let $I\subseteq \CC[a][x^\pm]$ be a parametrised polynomial ideal.  Let $\CC(a)\coloneqq \CC(a_j\mid j\in [m])$ denote the rational function field in the parameters $a_j$.  The \emph{generic specialisation} of $I$ is the ideal in $\CC(a)[x^\pm]$ generated by $I$, i.e.,
    \begin{equation*}
        I_{\CC(a)}\coloneqq \langle h\colon h \in I\rangle \subseteq \CC(a)[x^\pm].
    \end{equation*}
    The \emph{generic root count} of $I$ is $\ell_{I,\CC(a)}\coloneqq \dim_{\CC(a)}(\CC(a)[x^\pm]/I_{\CC(a)})\in\ZZ_{\geq 0}\cup\{\infty\}$.

    We say that $I$ is generically a complete intersection if $I_{\CC(a)}$ is a complete intersection, and that $I$ is generically zero-dimensional if $I_{\CC(a)}$ is zero-dimensional (in which case $\ell_{I,\CC(a)}<\infty$).
\end{definition}

\begin{remark}
    \label{rem:genericRootCount}
    The name `root count' for the vector space dimension $\ell_{I,P}$ is derived from the fact that it is the number of roots counted with a suitable algebraic multiplicity \cite[\S 4, Corollary~2.5]{CoxLittleOShea2005}. The name `generic root count' for the vector space dimension $\ell_{I,\CC(a)}$ is justified by the fact that there is an Zariski-open subset in the parameter space $U\subseteq\CC^{a}$ over which it is attained, i.e., $\ell_{I,P}=\ell_{I,\CC(a)}$ for all $P\in U$.
\end{remark}

\begin{example}
    Consider the parametrised principal ideal
    \begin{equation*}
        I\coloneqq\langle a_0+a_1x+a_2x^2\rangle\subseteq\CC[a_0,a_1,a_2][x^\pm].
    \end{equation*}
    Then $\ell_{I,(0,0,0)}=\infty$, $\ell_{I,(1,0,0)}=0$, $\ell_{I,(0,1,0)}=1$ and the generic root count is $\ell_{I,\CC(a)}=2$, which is attained whenever $a_2\neq 0$.
\end{example}

\subsection{Hadamard product of matroids}

Given matroids $M,N$ with shared ground set $E$,
we define the \emph{Hadamard product} $M \odot N$ to be the matroid with ground set $E$ where a non-empty set $F$ is independent if and only if
\begin{equation*}
    r_M(F') + r_N(F') \geq |F'| + 1 \qquad \text{for all } \quad \emptyset \subsetneq F' \subseteq F.
\end{equation*}
That $M \odot N$ is a matroid follows from classical results of Edmonds \cite{Edmonds2003} (which were later independently proven by Dunstan \cite{Dunstan1976});
see \cite[Theorem 2.7]{AntoliniDewarTanigawa2025} for an explicit proof.

In \cite{Bernstein2022},
Bernstein proved that the Hadamard product of the algebraic matroids for two linear spaces is equal to the algebraic matroid of the Hadamard product of those linear spaces.

\begin{theorem}[Bernstein {\cite[Theorem 3.1]{Bernstein2022}}]
\label{thm:bern_matroid_linear}
    Let $U,V \subseteq \mathbb{C}^E$ be linear spaces that are not contained in any coordinate hyperplane (equivalently, $M(U)$ and $M(V)$ are loopless).
    Then
    \begin{equation*}
        M(U \odot V) = M(U) \odot M(V).
    \end{equation*}
    In particular, the dimension of $U \odot V$ is equal to the rank of the matroid $M(U) \odot M(V)$.
\end{theorem}

A key step in Bernstein's proof is to describe the rank function of the Hadamard product of matroids using Bergman fans.

\begin{theorem}[Bernstein {\cite[Theorem 3.3]{Bernstein2022}}]
\label{thm:bern_tropical}
    Let $M,N$ be loopless matroids with a shared ground set $E$.
    Then a subset $F \subset E$ is independent in $M \odot N$ if and only if
    \begin{equation*}
        \dim \Bigl( \pi_F \bigl(\Trop (M) + \Trop(N)\bigr)\Bigr) = |F| +1,
    \end{equation*}
    where $\pi_F$ is the projection $(x_e)_{e \in E} \mapsto (x_e)_{e \in F}$.
\end{theorem}

\begin{remark}
    The matroid $M \odot N$ can be alternatively defined by the following procedure.
    First, one considers the monotone submodular function $F \mapsto r_M(F) + r_N(F) - 1$, for $F \subset E$.
    Taking the \emph{Dilworth truncation} (see \cite[Section~48]{Schrijver2003}) of the latter function yields a submodular, nonnegative function.
    Finally, using methods of Edmonds and Rota \cite{EdmondsRota1966}, one gets the desired matroid.
\end{remark}

\section{Flip products of matroids}
\label{flip_products}

In this section, we collect some useful observations about the flip product that are required throughout the paper.

First of all, now that we have introduced all the relevant notions, let us recall how we define the flip product between two matroids $M$ and $N$ over the same ground set $E$:
\begin{equation*}
    M \ast N \ := \ \bigl(-\Tropp(N)\bigr) \cdot \Tropp(M) \,.
\end{equation*}
We often use the equivalent definition for the flip product in terms of affine Bergman fans.
With this,
the flip product of $M$ and $N$ is the tropical intersection product
\begin{equation*}
    M \ast N = (-\Trop(N)) \cdot \Trop(M) \cdot \Trop (y_\epsilon -1)
\end{equation*}
for any arbitrary choice of $\epsilon \in E$.
Since $\Trop (y_\epsilon -1) = \mathbb{R}^{E \setminus \epsilon} \times \{0\}$ and $\Trop(M)$ contains $\mathbb{R} \cdot \bo$,
the flip product can be expressed as the cardinality of a set for some generic $w \in \mathbb{R}^E$:
\begin{equation}\label{eq:FPasCardinality}
    M \ast N = \Big| \Trop(M) \cap (-\Trop(N) +w) \cap \Trop (y_\epsilon -1) \Big| .
\end{equation}
By a direct inspection of how the stable intersection is defined, one has that $M \ast N = N \ast M$.

We begin with the following observation regarding matroids with loops.

\begin{proposition}\label{prop:loopbad}
    Let $M, N$ be matroids that share a ground set $E$.
    If either $M$ or $N$ contains a loop, then $M *N=0$.
\end{proposition}

\begin{proof}
    If $M$ say contains a loop, then $\Tropp(M) = \emptyset$ by definition.
    This implies that the stable intersection of $\Tropp(M)$ and $\Tropp(N)$ is empty, hence $(-\Tropp(N)) \cdot \Tropp(M) = 0$.
\end{proof}

We next identify exactly when the flip product of matroids is a finite positive integer.

\begin{proposition}\label{prop:intersection+number}
    Let $M,N$ be matroids that share a ground set $E$.
    \begin{enumerate}
        \item \label{prop:intersection+numberitem1} If $r(M) + r(N) < |E| +1$, then $M \ast N = 0$.
        \item \label{prop:intersection+numberitem2} If $r(M) + r(N) = |E| +1$, then $M \ast N$ is finite.
        \item \label{prop:intersection+numberitem3} If $r(M) + r(N) > |E| +1$, then $M \ast N$ is either 0 or $\infty$.
        \item \label{prop:intersection+numberitem4} $M \ast N$ is positive if and only if $M \odot N$ is a free matroid, i.e., all subsets of~$E$ are independent in $M \odot N$.
    \end{enumerate}
\end{proposition}

\begin{proof}
    The flip product $M \ast N$ is a positive integer if and only if $\Tropp(M) \wedge (-\Tropp(N))$ is zero-dimensional, or equivalently $\Trop(M) \wedge (-\Trop(N))$ is one-dimensional.
    Similarly, $M \ast N=0$ if and only if $\Tropp(M) \wedge (-\Tropp(N))$ is empty.
    Each of $\Trop(M)$ and $\Trop(N)$ have codimension $|E| - r(M)$ and $|E| - r(N)$ respectively, hence \Cref{thm:stable+intersection} gives that the stable intersection $\Trop(M) \wedge (-\Trop(N))$ is either a polyhedral complex of codimension $2|E| - r(M) - r(N)$ or it is empty.
    Moreover,
    since the intersection of the linearity spaces of $\Trop(M)$ and $\Trop(N)$ has dimension $\ell \geq 1$,
    if $2|E| - r(M) - r(N) + 1 > |E|$ then the stable intersection is empty.
    With this, we make three observations:
    \begin{itemize}
        \item If $r(M) + r(N) < |E| +1$, then $2|E| - r(M) - r(N) + 1 > |E|$ and so the stable intersection is empty (and so $M \ast N = 0$).
        \item If $r(M) + r(N) = |E| +1$,
        then either $\Trop(M) \wedge (-\Trop(N))$ is one-dimensional (and so $M \ast N$ is a positive integer) or $\Trop(M) \wedge (-\Trop(N))$ is empty (and so $M \ast N = 0$).
        \item If $r(M) + r(N) > |E| +1$, then $\Tropp(M) \wedge (-\Tropp(N))$ is either empty (and so $M \ast N = 0$) or has dimension greater than 1 (and so $M \ast N = \infty$). In either case, $M \ast N$ is not a positive integer.
    \end{itemize}
    These observations prove points \ref{prop:intersection+numberitem1}, \ref{prop:intersection+numberitem2} and \ref{prop:intersection+numberitem3} hold.

    We now proceed to show \ref{prop:intersection+numberitem4} holds.
    From our previous observations, we can assume that $r(M) + r(N) = |E| + 1$,
    and by \Cref{prop:loopbad} we can assume both $M$ and $N$ are loopless.
    By \Cref{lem:StableIntersectionNonempty}, the stable intersection is non-empty if and only if $\dim(\Trop(M) + (-\Trop(N))) = |E|$.
    By various properties of the affine span, we have
    \begin{align*}
        |E| &= \dim(\Trop(M) + (- \Trop(N)) \\
        &= \max_{\cF, \cG}(\dim(\sigma_\cF + (-\sigma_\cG))) \\
        &= \max_{\cF, \cG}(\dim(\aff(\sigma_\cF) + \aff(-\sigma_\cG))) \\
        &= \max_{\cF, \cG}(\dim(\aff(\sigma_\cF) + \aff(\sigma_\cG))) \\
        &= \dim(\Trop(M) + \Trop(N)).
    \end{align*}
    The result now follows from \Cref{thm:bern_tropical}.
\end{proof}

\begin{example}\label{ex:K4glueC4}
    Consider the graph $G =([6],E)$ on six vertices and nine edges obtained by gluing a copy of $K_4$ and a 4-cycle along an edge; see \Cref{fig:K4glueC4}.
    The corresponding graphic matroid $M_G$ has rank~$5$, hence taking $M = N = M_G$ satisfies the rank condition $r(M) + r(N) = |E| + 1$.
    However, consider $F = E(K_4) \subset E$: this subset has $r_{M_G}(F) = 3$,
    hence $|F| + 1 > r(M) + r(N)$.
    \Cref{prop:intersection+number} tells us that $M_G * M_G = 0$.
    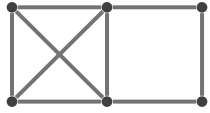
\begin{figure}[ht]
        \centering
        \begin{tikzpicture}[scale=1.25]
            \node[vertex] (a) at (0,0) {};
            \node[vertex] (b) at (1,0) {};
            \node[vertex] (c) at (1,1) {};
            \node[vertex] (d) at (0,1) {};
            \node[vertex] (e) at (2,0) {};
            \node[vertex] (f) at (2,1) {};
            \draw[edge] (a)--(b) (a)--(c) (a)--(d) (b)--(c) (b)--(d) (c)--(d);
            \draw[edge] (b)--(e) (e)--(f) (f)--(c);
        \end{tikzpicture}
        \caption{The graph $G$ formed from a $K_4$ glued to a $C_4$.}
        \label{fig:K4glueC4}
    \end{figure}
    Now consider the graph $G' =([6],E')$ obtained from the graph $G$ by moving a single edge; see \Cref{fig:multtri}.
    Taking $M = N = M_{G'}$, we have that $M \odot N$ is a free matroid and $r(M) + r(N) = |E'| + 1$.
    Hence, \Cref{prop:intersection+number} tells us that $M_{G'} * M_{G'}$ is a positive integer (in fact, $M_{G'} * M_{G'} = 16$).
    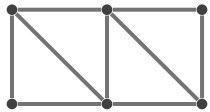
\begin{figure}[ht]
        \centering
        \begin{tikzpicture}[scale=1.25]
            \node[vertex] (a) at (0,0) {};
            \node[vertex] (b) at (1,0) {};
            \node[vertex] (c) at (1,1) {};
            \node[vertex] (d) at (0,1) {};
            \node[vertex] (e) at (2,0) {};
            \node[vertex] (f) at (2,1) {};
            \draw[edge] (a)--(b) (a)--(d) (b)--(c) (b)--(d) (c)--(d);
            \draw[edge] (b)--(e) (e)--(f) (f)--(c);
            \draw[edge] (c)--(e);
        \end{tikzpicture}
        \caption{The graph $G'$ formed from the graph $G$ by moving a single edge.}
        \label{fig:multtri}
    \end{figure}
\end{example}

\begin{example}\label{ex:flip_product_is_zero}
    Let $M$ be a non-uniform matroid on a ground set $E$.
    Since $M$ is non-uniform, then there exists a dependent set~$D$ whose cardinality equals the rank of~$M$.
    We define $N$ to be the matroid $U_{E\setminus D,k} \oplus U_{D,1}$, where $k = |E|-r(M)$. Note that
    \begin{equation*}
        r(M) + r(N) = r(M) + |E| - r(M) + 1 = |E| + 1
    \end{equation*}
    so $M\ast N$ is finite by \Cref{prop:intersection+number} \ref{prop:intersection+numberitem2}, and as $r_M(D)+r_N(D)<r+1=|D|+1$,
    by \Cref{prop:intersection+number}~\ref{prop:intersection+numberitem4} we have that $M\ast N=0$.
    Hence, for every non-uniform matroid $M$, there exists another loopless matroid $N$ on the same ground set $E$ such that $r(M) + r(N) = |E|+1$ and $M \ast N = 0$.
\end{example}

If our choice of matroids $M$ and $N$ have the same separation,
we can also prove that the flip product $M \ast N$ is zero.

\begin{lemma}\label{lem:zero_flip_if_disconnected}
    Let $M,N$ be matroids with shared ground set $E$. If $M \ast N$ is finite and there exists a set $\emptyset\neq F \subsetneq E$ such that
    \begin{equation*}
        M = (M \setminus F) \oplus (M \setminus (E \setminus F)) , \qquad N = (N \setminus F) \oplus (N \setminus (E \setminus F)),
    \end{equation*}
    then $M \ast N = 0$.
\end{lemma}

\begin{proof}
    As $M \ast N$ is finite, the result follows immediately from \Cref{prop:intersection+number} \ref{prop:intersection+numberitem1} and \ref{prop:intersection+numberitem3} if $r(M) + r(N) \neq |E| +1$. Now, assume that $r(M) + r(N) = |E| +1$. As, by hypothesis, $r(M)= r(M \setminus F) + r(M \setminus (E \setminus F)) = r_M(E\setminus F)+r_M(F)$ and $r(N)= r(N \setminus F) + r(N \setminus (E \setminus F)) = r_N(E\setminus F)+r_N(F)$, we have that
    \begin{equation*}
        r_M(E\setminus F)+r_M(F)+r_N(E\setminus F)+r_N(F)=|E|+1<(|E\setminus F|+1)+(|F|+1).
    \end{equation*}
    Therefore $r_M(E\setminus F)+r_N(E\setminus F)<|E\setminus F|+1$ or $r_M(F)+r_N(F)<|F|+1$, either implying $M\ast N=0$ by \Cref{prop:intersection+number} \ref{prop:intersection+numberitem4}.
\end{proof}

\begin{remark}\label{rmk:shared_coloop}
    If $M$ and $N$ share a coloop $e\in E$ and $|E|\geq2$, then \Cref{lem:zero_flip_if_disconnected} implies that $M \ast N = 0$, as in this case we have that $M\setminus(E \setminus e)=N\setminus(E \setminus e)=U_{\{e\},1}$, $M = (M \setminus e) \oplus U_{\{e\},1}$ and $N = (N \setminus e) \oplus U_{\{e\},1}$.
\end{remark}

When the ground set of both matroids contains exactly one element, it is easy to determine the exact value of the flip product.

\begin{lemma}\label{lem:1element}
    Let $M$ and $N$ be matroids over $E$.
    If $|E|=1$, then
    \begin{equation*}
        M \ast N =
        \begin{cases}
            1 &\text{if } M\cong U_{1,1} ~ \text{and} ~ N\cong U_{1,1}, \\
            0 &\text{otherwise}.
        \end{cases}
    \end{equation*}
\end{lemma}

\begin{proof}
    Up to swapping $M$ and $N$,
    there are three cases to consider:
    (i) neither $M$ nor $N$ are isomorphic to $U_{1,1}$,
    (ii) $M $ is not isomorphic to $U_{1,1}$ but $N$ is,
    and (iii) both $M$ and $N$ are isomorphic to $U_{1,1}$.
    The first two cases both imply $M$ contains a loop, which implies that $M \ast N = 0$ by \Cref{prop:loopbad}.
    If both $M$ and $N$ are isomorphic to $U_{1,1}$, then
    \begin{equation*}
        \Trop (M) = \Trop (N) = \mathbb{R} \quad \implies \quad \Trop (M) \cdot (-\Trop (N)) \cdot \Trop (y_\epsilon - 1) = 1,
    \end{equation*}
    and so $M \ast N = 1$.
\end{proof}

\section{Hadamard products for representable matroids}
\label{hadamard_products}

We recall from \Cref{sec:intro} that for linear spaces $U,V \subset \mathbb{C}^E$, the map~$f_{U,V}$ is the Hadamard map, the set $U \odot V$ is the Hadamard product, and $\sim$ is the equivalence relation on $U\times V$ given by the action of the algebraic torus.

It follows from \Cref{thm:bern_matroid_linear} that the value $\#(f_{U,V}^{-1}(\lambda)/\!\!\sim)$ for generic $\lambda \in \CC^E$ is a positive integer if and only if $f_{U,V}$ is dominant and
\begin{equation}\label{eq:rankuv}
    r\bigl(M(U)\bigr) + r\bigl(M(V)\bigr) = |E|+1.
\end{equation}
In particular,
if $f_{U,V}$ is not dominant then $\#(f_{U,V}^{-1}(\lambda)/\!\!\sim) = 0$,
whilst if $f_{U,V}$ is dominant but \Cref{eq:rankuv} does not hold then
$r(M(U)) + r(M(V)) > |E|+1$ and $\#(f_{U,V}^{-1}(\lambda)/\!\!\sim) = \infty$.

\subsection{Proof of \texorpdfstring{\Cref{thm:mainalg}}{Hadamard product theorem}}

We first reframe enumerating generic fibres as a generic root count.
Recall that given a linear space $U \subseteq \CC^E$, for each circuit $C$ of $M(U)$, there exists a unique, up to scalar, linear polynomial $h_{C,U} \in \mathbb{C}[x_e \mid e \in E]$ with support equal to $C$ that vanishes in $U$ \cite[Lemma~4.1.4]{MaclaganSturmfels2015}.
Explicitly, if $U$ is the column span of a matrix $A$, then for each circuit $C \in M(U)$, the rows $\{A_e\}_{e \in C}$ of $A$ indexed by $C$ are minimally linearly dependent, i.e., there exist unique (up to a scalar factor) non-zero coefficients $\alpha_e \in \CC$  such that $\sum_{e \in C} \alpha_e A_e = 0$ and we define $h_{C,U} = \sum_{e \in C} \alpha_e x_e$.
Furthermore, these polynomials completely determine $U$, since
\begin{equation*}
    U = \mathbb{V}\bigl( \{h_{C,U} \colon C \text{ is a circuit of } M(U) \} \bigr) \, ,
\end{equation*}
where $\mathbb{V}(S)$ denotes the algebraic set generated by $S$.

\begin{lemma}
\label{lem:hadamardrealisationNumberTwoDimensionalEdgeVariables}
    Let $U, V \subset \mathbb{C}^E$ be linear subspaces not contained in any coordinate hyperplane and fix some $\epsilon \in E$.
    Consider the parametrised Laurent polynomial ring
    \begin{align*}
        \CC[a,b][u^{\pm },v^{\pm}] := \CC \Big[a_{e}\mid e \in E \Big] \Big[ b_1,b_2\Big] \left[ u^{\pm1}, v^{\pm 1} \mid e\in E \right]
    \end{align*}
    with parameters $a_{e}, b_{\ell}$ and variables $u_e,v_e$.
    Let $I_{U,V}$ be the parametrised polynomial ideal generated by
    \begin{equation*}
        \begin{aligned}
            f_{e}&\coloneqq u_e\cdot v_e - a_{e} &&\text{for each } e \in E \\
            g &\coloneqq b_1 u_\epsilon + b_2 v_\epsilon, \\
            h_{C,U} &\in \mathbb{C}[u_e^\pm \mid e \in E] &&\text{for each circuit $C$ of } M(U) , \\
            h_{C,V} &\in \mathbb{C}[v_e^\pm \mid e \in E] &&\text{for each circuit $C$ of } M(V) .
        \end{aligned}
    \end{equation*}
    Then the generic root count $\ell_{I_{U,V},\CC(a,b)}$ is equal to $2 \cdot\#(f_{U,V}^{-1}(\lambda)/\!\!\sim)$ for any choice of generic $\lambda \in \mathbb{C}^E$.
\end{lemma}

\begin{proof}
    As mentioned in \Cref{rem:genericRootCount}, we can choose generic $(\alpha,\beta)\in \CC^{E}\times\CC^2$ so that $\ell_{I_{U,V},\CC(a,b)}=\ell_{I_{U,V},(\alpha,\beta)}$.
    In particular, we can assume that $(\alpha,\beta)$ has nonzero coordinates.
    It is clear that the collection of linear polynomials $h_{C,U},h_{C,V}$ cut out the set $U \times V$ in $\CC^E \times \CC^E$.
    When combined with the collection of polynomials $f_{e,(\alpha,\beta)}$, this describes the fibre $f_{U,V}^{-1}(\alpha)$.
    Hence $\mathbb{V}(I_{U,V,(\alpha,\beta)})=f_{U,V}^{-1}(\alpha)\cap \mathbb{V}(g_{(\alpha,\beta)})$.

    We first note that $2 \#(f_{U,V}^{-1}(\alpha)/\!\!\sim)=\#\mathbb{V}(I_{U,V,(\alpha,\beta)})$.
    Indeed, if $(x,y)\in f_{U,V}^{-1}(\alpha)$, we have for any $t\in\CC \setminus \{0\}$ that $(tx,t^{-1}y)\in \mathbb{V}(I_{U,V,(\alpha,\beta)})$ if and only if $(tx,t^{-1}y)\in \mathbb{V}(g_{(\alpha,\beta)})$, that is, $t \beta_1 x_\epsilon + t^{-1}\beta_2 y_\epsilon = 0$.
    As $(\alpha,\beta)$ has nonzero coordinates, the latter happens if and only if $t^2=-\beta_2y_\epsilon/(\beta_1x_\epsilon)$, which has exactly 2 solutions generically.
    Thus, each equivalence class of points in $f_{U,V}^{-1}(\alpha)$ has exactly two representatives in $\mathbb{V}(I_{U,V,(\alpha,\beta)})$.

    Now, we want to show that $\ell_{I_{U,V},(\alpha,\beta)}=\#\mathbb{V}(I_{U,V,(\alpha,\beta)})$.
    This is automatic if $\#\mathbb{V}(I_{U,V,(\alpha,\beta)}) \in \{0, \infty\}$,
    so we need only consider the case where $\#\mathbb{V}(I_{U,V,(\alpha,\beta)})$ is a positive integer.
    Here, we need to show that every point in $\mathbb{V}(I_{U,V,(\alpha,\beta)})$ has multiplicity one.
    By \cite[Chapter~4, Corollary~2.6]{CoxLittleOShea2005}, it is sufficient to prove that each point in $\mathbb{V}(I_{U,V,(\alpha,\beta)})$ has full rank Jacobian.
    In fact, since $\mathbb{V}(I_{U,V,(\alpha,\beta)}) = f_{U,V}^{-1}(\alpha)\cap \mathbb{V}(g_{(\alpha,\beta)})$,
    it suffices to show that the intersection of tangent space of~$f_{U,V}^{-1}(\alpha)$ and $\mathbb{V}(g_{(\alpha,\beta)})$ at each point $(x,y) \in \mathbb{V}(I_{U,V,(\alpha,\beta)})$ is zero-dimensional.

    Choose any $(x,y)\in \mathbb{V}(I_{U,V,(\alpha,\beta)})\subseteq\CC^E\times\CC^E$.
    By \Cref{thm:bern_matroid_linear}, the map $f_{U,V}$ is dominant and \Cref{eq:rankuv} holds.
    Taking into account the natural action of the algebraic torus on $U \odot V$ and the genericity of $(x,y)$ (the latter being a consequence of the genericity of $(\alpha,\beta)$),
    this implies that the kernel of the derivative $Df_{U,V}(x,y) \colon U \times V \rightarrow \mathbb{C}^E$ is the linear space $\left\langle (x,-y) \right\rangle$.
    Hence, the tangent space of $f_{U,V}^{-1}(\alpha)$ at $(x,y)$ is $\langle (x,-y) \rangle$.
    Since $\beta_1 x_\epsilon + \beta_2 y_\epsilon = 0$ and (by genericity)
    we have that $\beta_1 x_\epsilon - \beta_2 y_\epsilon \neq 0$.
    It hence follows that the tangent space for $\mathbb{V}(g_{(\alpha,\beta)})$ at $(x,y)$ does not contain the linear space $\langle (x,-y) \rangle$,
    and so the intersection of tangent spaces is zero-dimensional.
    This now concludes the proof.
\end{proof}

We can then reformulate this into a more `tropical-friendly' format:

\begin{lemma}
    \label{lem:hadamardSIAGA2}
    Let $U, V \subset \mathbb{C}^E$ be linear subspaces not contained in any coordinate hyperplane and fix some $\epsilon \in E$.
    Consider the modified parametrised Laurent polynomial ring
    \begin{equation*}
        \CC \Big[a_{e}\mid e \in E \Big] \Big[ b_1,b_2\Big] \Big[ y_{e}^{\pm1}\mid e\in E \Big],
    \end{equation*}
    with parameters $a_e, b_{l}$ and variables $y_{e}$.
    Let $I_{(G,H)}'$ be the parametrised polynomial ideal generated by:
    \begin{align*}
        \label{eq:twoDimensionalModification2}
        h_{C,U}'&\coloneqq \ell_{C,M(U)}(y) &&\text{for each circuit $C$ of } M(U) , \\
        h_{C,V}'&\coloneqq \ell_{C,M(V)}((a_e y^{-1}_e)) &&\text{for each circuit $C$ of } M(V), \\
        g' &\coloneqq b_1y_{\epsilon}^2 + b_2.
    \end{align*}
    Then $\ell_{I_{U,V}',\CC(a,b)}$ is equal to $2 \cdot \#(f_{U,V}^{-1}(\lambda)/\!\!\sim)$ for any choice of generic $\lambda \in \mathbb{C}^E$.
\end{lemma}

\begin{proof}
    Using that $f_e=0$ implies $v_e = a_e u_e^{-1}$,
    we make the substitution $u_e \mapsto y_e$ and $v_e \mapsto a_ey_e^{-1}$ in the ideal $I_{U,V}$ for each $e \in E$.
\end{proof}

We now prove our first main theorem, which we recall is the following.

\mainalg*

\begin{proof}
    The proof proceeds almost identically to the specific case outlined in \cite[Theorem 3.8]{ClarkeDewarEtAl2025} when $M=N=M_G$ for some graph $G$.
    We now provide a short proof outlining the key details.

    If either of $U$ or $V$ are contained in a coordinate hyperplane then $\#(f_{U,V}^{-1}(\lambda)/\!\!\sim) = 0$ and (since one of $M(U)$ or $M(V)$ contain a loop) $M(U) \ast M(V)=0$,
    so we may suppose this is not the case.
    By \Cref{thm:bern_matroid_linear} and \Cref{prop:intersection+number},
    we may assume that $\#(f_{U,V}^{-1}(\lambda)/\!\!\sim)$ is finite and $r(M)+r(N) = |E| +1$.
    Consider the ideals from \Cref{lem:hadamardSIAGA2}:
    \begin{align*}
        I_1 &= \langle h_{C,M(U)}'\mid C\subseteq E \text{ circuit of } M(U) \rangle, &
        I_2 &= \langle h_{C,M(V)}'\mid C\subseteq E \text{ circuit of } M(V)\rangle, &
        I_3 &= \langle g' \rangle \, .
    \end{align*}
    Using the exact same techniques as described in \cite[Theorem 3.8]{ClarkeDewarEtAl2025},
    we see that their tropical varieties are
    \begin{equation*}
        \Trop(I_1) = \Trop\bigl(M(U)\bigr), \qquad 
        \Trop(I_2) = -\Trop\bigl(M(V)\bigr), \qquad 
        \Trop(I_3) = \Trop(y^2_\epsilon - 1).
    \end{equation*}
    It is immediate that $\codim(\Trop(y^2_\epsilon - 1)) = 1$.
    As $\dim(\Trop(M)) = r(M)$ for any matroid $M$ and tropicalisation preserves dimension (see, for example, \cite[Theorem 3.3.5]{MaclaganSturmfels2015}), we deduce that
    \begin{equation*}
        \codim(I_{1,\CC(a,c)}) + \codim(I_{2,\CC(a,c)}) + \codim(I_{3,\CC(a,c)}) = \bigl(|E|- r(M)\bigr) + \bigl(|E|- r(N)\bigr) + 1 = |E|.
    \end{equation*}
    Using the same method described in \cite[Theorem 3.8]{ClarkeDewarEtAl2025},
    it can be checked that $I_1,I_2,I_3$ satisfy the necessary conditions to apply \cite[Proposition~1]{HoltRen2025},
    and so
    \begin{equation*}
        \ell_{I'_{U,V}, \mathbb{C}(a,b)} = 
        \Trop(I_{1,P}) \cdot \Trop(I_{2,P}) \cdot \Trop(I_{3,P}) 
        = \Trop\bigl(M(U)\bigr) \cdot \Bigl(-\Trop\bigl(M(V)\bigr)\Bigr)\cdot \Trop(y_{\epsilon}^2-1)
    \end{equation*}
    for generic $P \in \CC^{E} \times \mathbb{C}^2$.
    The statement then follows from \Cref{lem:hadamardSIAGA2} and the fact that $\Trop(y_{\epsilon}^2-1)$ and $\Trop(y_{\epsilon}-1)$ coincide set-theoretically, but the one cell of $\Trop(y_{\epsilon}^2-1)$ is of multiplicity $2$, whereas the cell of $\Trop(y_{\epsilon}-1)$ has multiplicity $1$.
\end{proof}

\Cref{thm:mainalg} can be simplified in the special case where both $U$ and $V$ are generic linear spaces (i.e., generic points in their corresponding Grassmannians):

\begin{corollary}\label{cor:mainalg}
    Let $U, V \subset \mathbb{C}^E$ be generic linear subspaces.
    If $\dim U + \dim V = |E|+1$,
    then for any generic $\lambda \in \mathbb{C}^E$ we have
    \begin{equation*}
        \#(f_{U,V}^{-1}(\lambda)/\!\!\sim) = \binom{|E|-1}{\dim U - 1} = \binom{|E|-1}{\dim V - 1}.
    \end{equation*}
\end{corollary}

\begin{proof}
    As $U$ and $V$ are generic,
    we have that $M(U) = U_{E,\dim U}$ and $M(V) = U_{E,\dim V}$,
    where $U_{E,r}$ is the uniform matroid with ground set $E$ and rank $r$.
    By \Cref{thm:nbc}, we have that $M(U) \ast M(V)$ is equal to the number of nbc-bases of $M(U)$ (equivalently, of $M(V)$).
    It is easy to check that $\nbc(U_{E,r}) = \binom{|E|-1}{r-1}$, which concludes the result.
\end{proof}

\begin{remark}
    Note that, by \cite[Theorem 5.1, Corollary 6.2]{Speyer2009}, it was known that if $\widetilde{U}$ is a generic linear subspace of $\CC^{E\sqcup\{a\}}$ of dimension $d$ and $\widetilde{V}=\widetilde{U}^\perp$, then the cardinality of a fibre of a generic point in their Hadamard product is $\beta(U_{E\sqcup\{a\},d})=\binom{|E|-1}{d-1}$ (see \Cref{matroid_applications}).
    By projecting away the coordinate $a$, namely, by considering $\pi \colon \CC^{E\sqcup\{a\}}\longrightarrow\CC^E$, we get the situation of our \Cref{cor:mainalg} with $U=\pi(\widetilde{U})$ and $V=\pi(\widetilde{V})$. Hence, \Cref{cor:mainalg} may be viewed as a generalisation of this fact, as $U$ and $V$ are now not required to be related to each other.
    \Cref{cor:mainalg} may also be viewed as a corollary of \cite[Theorem 5.2]{KummerVinzant2019}, by using a similar trick: if $\widetilde{U},\widetilde{V}\subseteq\CC^{E\sqcup\{a\}}$ are (unrelated) generic linear subspaces of complementary dimensions, then the projection
    \[
        \pi \colon \mathbb{P}\left(\widetilde{U}\odot\widetilde{V} \right) \subset\PP(\CC^{E\sqcup\{a\}})\dashrightarrow\PP(\CC^E)
    \]
    is a rational map is of degree $\binom{|E|-1}{d-1}$, which in turn implies that the ``projective Hadamard map'' $\PP\bigl(\pi(\widetilde{U})\bigr) \times  \PP\bigl(\pi(\widetilde{V})\bigr)\dashrightarrow \PP(\CC^E)$ is also of degree $\binom{|E|-1}{d-1}$.
\end{remark}

\subsection{Reformulating as squared distance map}
\label{subsec:gandh}

We now reformulate \Cref{thm:mainalg} into a format that is directly applicable for some specific situations (see \Cref{applications}).

Let $X, Y$ be complex matrices of size $|E| \times n_X$ and $|E| \times n_Y$ respectively,
each with rows $X_e \in \mathbb{C}^{1 \times n_X}$, $Y_e \in \mathbb{C}^{1 \times n_Y}$ for each $e \in E$.
From this, define the maps
\begin{align*}
    g_{X,Y} &\colon \frac{\mathbb{C}^{n_X}}{\ker X} \times \frac{\mathbb{C}^{n_Y}}{\ker Y} \longrightarrow \mathbb{C}^E,  ~ (x,y) \longmapsto \left( (X_e x) \cdot (Y_ey) \right)_{e \in E}, \\
    h_{X} &\colon \frac{\mathbb{C}^{n_X}}{\ker X} \times \frac{\mathbb{C}^{n_X}}{\ker X} \longrightarrow \mathbb{C}^E, ~ (x,y) \longmapsto \left( (X_e x)^2 + (X_ey)^2 \right)_{e \in E}.
\end{align*}
We fix the equivalence relations $\sim_g,\sim_h$ such that:
\begin{enumerate}
    \item $(x,y) \sim_g (x',y')$ if and only if there exists $r \in \mathbb{C} \setminus \{0\}$ where $x' = rx$ and $y' = \frac{1}{r}y$;
    \item $(x,y) \sim_h (x',y')$ if and only if there exists a $2 \times 2$ matrix $A$ with $A^T A = A A^T = I_2$ where for each $v \in [n_X]$ we have
    \begin{equation*}
        A
        \begin{pmatrix}
            x(v) \\
            y(v)
        \end{pmatrix}
        =
        \begin{pmatrix}
            x'(v) \\
            y'(v)
        \end{pmatrix}.
    \end{equation*}
\end{enumerate}
Simply put, the relation $\sim_g$ describes equivalence by torus action, while $\sim_h$ describes equivalence by rotation and/or reflection.

\begin{theorem}\label{thm:gmap}
    Let $X, Y$ be complex matrices of size $|E| \times n_X$ and $|E| \times n_Y$ respectively,
    and fix $U,V$ to be the column spaces of $X$ and $Y$ respectively.
    Then for any generic $\lambda \in \mathbb{C}^E$,
    we have
    \begin{equation*}
        \# \left( g_{X,Y}^{-1}(\lambda)/\!\!\sim_g \right) = M(U) \ast M(V).
    \end{equation*}
\end{theorem}

\begin{proof}
    By \Cref{thm:mainalg},
    it suffices to show that
    \begin{equation}\label{eq:fisg}
        \# \left( g_{X,Y}^{-1}(\lambda)/\!\!\sim_g \right) = \# \left( f_{U,V}^{-1}(\lambda)/\!\!\sim \right)
    \end{equation}
    for generic $\lambda \in \mathbb{C}^E$. Observe that the map
    \begin{equation*}
        \phi \colon \frac{\mathbb{C}^{n_X}}{\ker X} \times \frac{\mathbb{C}^{n_Y}}{\ker Y} \longrightarrow U \times V, ~ (x,y) \longmapsto (X x , Y y)
    \end{equation*}
    is a linear isomorphism such that  $g_{X,Y} = f_{U,V} \circ \phi$.
    Also, the equivalence classes under $\sim_g$ are mapped to equivalence classes under $\sim$.
    Hence, $\phi$ defines a bijection between $g_{X,Y}^{-1}(\lambda)/\sim_g$ and~$f_{U,V}^{-1}(\lambda)/\sim$.
\end{proof}

\begin{theorem}\label{thm:hmap}
    Let $X$ be a $|E| \times n_X$ complex matrix,
    and fix $U$ to be the column space of $X$.
    Then for any generic $\lambda \in \mathbb{C}^E$,
    we have
    \begin{equation*}
        \# \left( h_{X}^{-1}(\lambda)/\!\!\sim_h \right) = \frac{1}{2} (M(U) \ast M(U)).
    \end{equation*}
\end{theorem}

\begin{proof}
    Given $\mathfrak{i}$ is the imaginary unit of $\mathbb{C}$,
    fix the linear isomorphism
    \begin{equation*}
        \begin{array}{rccc}
        \rho \colon & \displaystyle \frac{\mathbb{C}^{n_X}}{\ker X} \times \frac{\mathbb{C}^{n_X}}{\ker X} & \longrightarrow & \displaystyle \frac{\mathbb{C}^{n_X}}{\ker X} \times \frac{\mathbb{C}^{n_X}}{\ker X}, \\[8pt]
        & (x,y) & \longmapsto & (x + \mathfrak{i} y, x -\mathfrak{i} y)
        \end{array}
    \end{equation*}
    We observe that $g_{X,X} \circ \rho = h_X$, and $(x,y) \sim_f (x',y')$ if and only if either $\rho(x,y) \sim_g \rho(x',y')$ or $\rho(x,y) \sim_g \rho(y',x')$.
    Hence,
    \begin{equation*}
        \# \left( h_{X}^{-1}(\lambda)/\!\!\sim_h \right) = \frac{1}{2} \cdot \# \left( g_{X,X}^{-1}(\lambda)/\!\!\sim_g \right).
    \end{equation*}
    The result now follows from \Cref{thm:gmap}.
\end{proof}

\section{Inductive algorithm for flip products}
\label{algorithm}

In this section, we prove our inductive algorithm: \Cref{thm:main}.
For this, we require the following two lemmas that we prove later in \Cref{subsec:defferedlemmas}:

\begin{restatable}[Coloop removal]{lemma}{removebridge}\label{lem:removebridge}
    Let $M,N$ be matroids that share a ground set $E$.
    If $r(M) + r(N) = |E|+1$ and $e$ is a coloop of $N$ but neither a coloop or a loop of $M$,
    then
    \begin{equation*}
        M \ast N = (M \setminus e) \ast (N \setminus e).
    \end{equation*}
\end{restatable}

\begin{restatable}[Tropical untangling]{lemma}{tropicaluntangling}\label{lemma:tropical_untangling}
    Let $E_1$ and $E_2$ be finite sets with $E_1\cap E_2=\{\epsilon\}$.
    Let $M_1,\ M_2,\ N_1$ and $N_2$ be matroids with ground sets as follows:
    \begin{center}
        \begin{tabular}{rcccc}
            \toprule
            matroid & $M_1$ & $M_2$ & $N_1$ & $N_2$ \\
            ground set & $E_1$ & $E_2 \setminus \epsilon$ & $E_1 \setminus \epsilon$ & $E_2$ \\
            \bottomrule
        \end{tabular}
    \end{center}
    Assume that
    \begin{enumerate}
        \item $M_1$ and $N_2$ are loopless;
        \item $\epsilon$ is not a coloop of $M_1$ or $N_2$;
        \item $r(M_1\setminus \epsilon)+r(N_1)=|E_1|$;
        \item $r(M_2)+r(N_2\setminus \epsilon)=|E_2|$.
    \end{enumerate}
    Then
    \begin{equation*}
        (M_1  \oplus  M_2) \ast (N_1  \oplus  N_2) = \Big( (M_1 \setminus \epsilon) \ast N_1 \Big) \Big( M_2 \ast (N_2 \setminus \epsilon) \Big).
    \end{equation*}
\end{restatable}

\begin{remark}
    \Cref{lem:removebridge} and \Cref{lemma:tropical_untangling} are the tropical counterparts of \cite[Lemma 2.19]{CapcoGalletEtAl2018} and \cite[Proposition 2.21]{CapcoGalletEtAl2018} respectively.
\end{remark}

For the rest of the section, we use the convention that $0\cdot\infty=0$ and we fix the following notation.
For any loopless matroids $M,N$ on the ground set $E$ and any vector $\lambda\in\RR^E$, we set
\begin{equation*}
     (M \wedge N)_\lambda := (\lambda-\Trop(N)) \wedge\Trop(M) \wedge \Trop(y_\epsilon - 1).
\end{equation*}
By \cite[Proposition~3.6.12]{MaclaganSturmfels2015}, the $(M \wedge N)_\lambda$ contains exactly $M \ast N$ points when each point is counted with multiplicity. Moreover, if $\lambda$ is generic, then every point of $(M \wedge N)_\lambda$ has multiplicity~$1$ and the cardinality of the set is exactly $M \ast N$.
For any point $x \in \mathbb{R}^E$, any $e \in E$ and any $t \in \mathbb{R}$,
we fix
\begin{equation*}
    E_{x>t} := \{ f \in E : x(f)> t\}, \qquad E_{x<t} := \{ f \in E : x(f) < t\}.
\end{equation*}

\subsection{Proof of \texorpdfstring{\Cref{thm:main}}{algorithm theorem}}

For the following result, we recall that $\Star_x(X)$ denotes the star of a tropical cycle at a point~$x$ and $M_{p,w}$ denotes the matroid generated by the tropical Pl\"{u}cker coordinates $p$ and weight $w$;
see \Cref{def: matroid subdivision} and \Cref{tropical_background} for definitions and more details.

\begin{lemma}\label{lemma:point_multiplicity}
    Let $M,N$ be loopless matroids with shared ground set $E$ with $M \ast N \in \mathbb{Z}_{>0}$.
	Choose any $\epsilon \in E$ and any $\lambda \in \mathbb{R}^E$.
	Given points $x \in \mathbb{R}^E$ and $y \in \mathbb{R}^E$ satisfying $x+y =\lambda$ and $x(\epsilon) = 0$,
	the multiplicity\footnote{Here multiplicity 0 would mean that $x\notin(M \wedge N)_\lambda$.} of the point $x$ in $(M \wedge N)_\lambda$ is
	\begin{equation*}
		\left( \bigoplus_{t \in \im x} M \setminus E_{x<t} \, / \, E_{x>t} \right) \ast \left( \bigoplus_{t \in \im y} N \setminus E_{y<t} \, / \, E_{y>t} \right).
	\end{equation*}
\end{lemma}

\begin{proof}
    As $(M \wedge N)_\lambda$ is zero dimensional (\Cref{prop:intersection+number}), we have that $\Star_{x}\bigl((M \wedge N)_\lambda\bigr)$ consists solely of the zero vector $\mathbf{0}$ with multiplicity equal to the multiplicity of $x$ in $(M \wedge N)_\lambda$; here, we allow the point $x$ to be in $(M \wedge N)_\lambda$ with multiplicity zero. 
    By \cite[Lemma~3.6.7]{MaclaganSturmfels2015} and \cite[Corollary~4.4.8]{MaclaganSturmfels2015}, we have
    \begin{align*}
        \Star_{x}\bigl((M \wedge N)_\lambda\bigr) &= \Star_{x}\bigl(\Trop(M)\bigr) \wedge \Star_{x}\bigl(\lambda-\Trop(N)\bigr) \wedge \Star_{x}\bigl(\Trop(y_\epsilon-1)\bigr) \\
        &= \Star_{x}\bigl(\Trop(M)\bigr) \wedge \Bigl(-\Star_{\lambda-x}\bigl(\Trop(N)\bigr)\Bigr) \wedge \Star_{x}\bigl(\Trop(y_\epsilon-1)\bigr) \\
        &= \Trop(M_{\mathbf{0},x}) \wedge \bigl(- \Trop(N_{\mathbf{0},y})\bigr) \wedge \Trop(y_\epsilon-1).
    \end{align*}
    Finally, \cite[Proposition~2]{ArdilaKlivans2006} tells us that
    \begin{equation*}
        M_{\mathbf{0},x}=\bigoplus_{t \in \im x} M \setminus E_{x<t} \, / \, E_{x>t}, \qquad N_{\mathbf{0},y}=\bigoplus_{t \in \im y} N \setminus E_{y<t} \, / \, E_{y>t},
    \end{equation*}
    and so $x$ has the desired multiplicity in $(M \wedge N)_\lambda$.
\end{proof}

\begin{lemma}\label{lem:coloop case}
    Let $M,N$ be loopless matroids on $E$ such that $r(M)+r(N)=|E|+1$.
    Let $\epsilon\in E$ and $E_1,E_2\subseteq E$ be such that $E_1\cup E_2=E$ and $E_1\cap E_2=\{\epsilon\}$.
    If $|E_1|\geq2$, $|E_2|\geq2$ and $\epsilon$ is a coloop of either $M\setminus(E_2\setminus\epsilon)$ or $N\setminus(E_1\setminus\epsilon)$ or both, then
    \begin{equation*}
        \bigl(M \setminus (E_2 \setminus \epsilon)  \oplus  M / E_1\bigr) \ast \bigl(N \setminus (E_1 \setminus \epsilon)  \oplus  N / E_2\bigr) = 0.
    \end{equation*}
\end{lemma}

\begin{proof}
    First assume that $\epsilon$ is a coloop of exactly one, say $M\setminus(E_2 \setminus \epsilon)$.
    Then by \Cref{lem:removebridge} we have that
    \begin{equation*}
        \big(M \setminus (E_2 \setminus \epsilon)  \oplus  M / E_1\big) \ast \big(N \setminus (E_1 \setminus \epsilon)  \oplus  N / E_2\big) = \big(M \setminus E_2  \oplus  M / E_1\big) \ast \big(N \setminus E_1  \oplus  N / E_2\big)
    \end{equation*}
    and the right hand side is 0 by \Cref{lem:zero_flip_if_disconnected}.
    If, on the other hand, $\epsilon$ is a coloop of both $M\setminus(E_2 \setminus \epsilon)$ and $N\setminus(E_1\setminus\epsilon)$,
    then the desired equality follows from \Cref{rmk:shared_coloop}.
\end{proof}

\begin{lemma}\label{lem:coloop case2}
    Let $M,N$ be loopless matroids on $E$ such that $r(M)+r(N)=|E|+1$.
    Let $\epsilon\in E$ and $E_1,E_2\subseteq E$ be such that $E_1\cup E_2=E$ and $E_1\cap E_2=\{\epsilon\}$.
    If $|E_1|\geq2$, $|E_2|\geq2$ and $\epsilon$ is a coloop of either $M\setminus(E_2\setminus\epsilon)$ or $N\setminus(E_1\setminus\epsilon)$ or both, then
    \begin{equation*}
        \Big( (M \setminus E_2) \ast (N /E_2) \Big) \Big( (M / E_1) \ast (N \setminus  E_1) \Big) =0.
    \end{equation*}
\end{lemma}

\begin{proof}
    Define the constants
    \begin{align*}
    c_1 :=
    \begin{cases}
        1 &\text{if } \epsilon ~ \text{is a coloop of } M\setminus(E_2 \setminus \epsilon), \\
        0 &\text{otherwise},
    \end{cases}\qquad
     c_2 :=
    \begin{cases}
        1 &\text{if } \epsilon ~ \text{is a coloop of } N\setminus(E_1\setminus\epsilon), \\
        0 &\text{otherwise}.
    \end{cases}
    \end{align*}
    By our assumption, $c_1+c_2>0$.
    Then
    \begin{align*}
        & \ r(M \setminus E_2) + r (N /  E_2)+r (M / E_1) + r (N \setminus  E_1)\\
        =
        &\ r(M \setminus (E_2\setminus\epsilon)) +r (M / E_1) + r (N /  E_2) + r (N \setminus  (E_1\setminus\epsilon))-(c_1+c_2)\\
        =
        &\ r_M(E_1)+ \Big(r(M)-r_M(E_1) \Big)+ \Big(r(N)-r_N(E_2) \Big)+r_N(E_2)-(c_1+c_2)\\
        =
        &\ r(M) + r(N) -(c_1+c_2)\\
        <
        &\ |E|+1.
    \end{align*}
    Since $|E_1|+|E_2| = |E|+1$,
    either $r (M \setminus E_2) + r (N /  E_2) < |E_1|$ or $r (M / E_1) + r (N \setminus  E_1) < |E_2|$ or both.
    The result now follows from \Cref{prop:intersection+number} \ref{prop:intersection+numberitem1}.
\end{proof}

\begin{lemma}\label{lem:cases_untangling}
    Let $M,N$ be loopless matroids on $E$ such that $r(M)+r(N)=|E|+1$.
    Let $\epsilon\in E$ and $E_1,E_2\subseteq E$ be such that $E_1\cup E_2=E$ and $E_1\cap E_2=\{\epsilon\}$.
    Then
    \begin{equation}\label{eq:untangling}
    \begin{gathered}
        \Big(M \setminus (E_2 \setminus \epsilon)  \oplus  M / E_1\Big) \ast \Big(N \setminus (E_1 \setminus \epsilon)  \oplus  N / E_2\Big)\\
        = 
        \Big( (M \setminus E_2) \ast (N /E_2) \Big) \Big( (M / E_1) \ast (N \setminus  E_1) \Big).
    \end{gathered}
    \end{equation}
\end{lemma}

\begin{proof}
    By \Cref{lem:coloop case,lem:coloop case2},
    we need only consider the remaining cases where either $|E_1|=1$, $|E_2|=1$, or $\epsilon$ is not a coloop of $M\setminus(E_2\setminus\epsilon)$ and also not a coloop $N\setminus(E_1\setminus\epsilon)$.
    
    If $\epsilon$ is not a coloop of either $M\setminus(E_2\setminus\epsilon)$ or $N\setminus(E_1\setminus\epsilon)$, then we immediately get the result by \Cref{lemma:tropical_untangling} with the following identifications:
    \begin{equation*}
        M_1=M\setminus(E_2 \setminus \epsilon), \qquad M_2=M / E_1, \qquad N_1=N / E_2, \qquad N_2=N \setminus (E_1 \setminus \epsilon).
    \end{equation*}
    Now suppose without loss of generality that $|E_1|=1$.
    Then $E_1=\{\epsilon\}$ and $E_2=E$. Recalling that the flip product of two empty matroids is by convention $1$, \Cref{eq:untangling} that we are aiming to prove becomes
    \begin{equation}\label{eq:single_untangling}
        \big(M\setminus(E \setminus \epsilon)\oplus M/\epsilon\big) \ast N = (M/\epsilon) \ast (N\setminus \epsilon).
    \end{equation}
    Note that $M\setminus(E \setminus \epsilon)=U_{\{\epsilon\},1}$ as $M$ is loopless, and so $\epsilon$ is a coloop of $M\setminus(E_2 \setminus \epsilon)\oplus M/\epsilon$.
    If $\epsilon$ is not a coloop of $N$ then \Cref{eq:single_untangling} is true by \Cref{lem:removebridge}.
    If $\epsilon$ is a coloop of $N$,
    then the left hand side of \Cref{eq:single_untangling} is 0 by \Cref{rmk:shared_coloop} and,
    since
    \begin{equation*}
        r(M/\epsilon) + r(N \setminus \epsilon) = \Big( r(M) -1 \Big) + \Big( r(N ) - 1 \Big) = |E| < |E| +1,
    \end{equation*}
    the right hand side of \Cref{eq:single_untangling} is $0$ by \Cref{prop:intersection+number} \ref{prop:intersection+numberitem1}.
\end{proof}

We are now ready to prove \Cref{thm:main},
which we recall is the following:

\main*

\begin{proof}
    By \Cref{lemma:point_multiplicity}, we have
    \begin{equation*}
        M \ast N = \sum_{x \in (M \wedge N)_\lambda} \mult_{(M \wedge N)_\lambda}(x) = \sum_{\substack{x\in \mathbb{R}^E \\ x(\epsilon)=0\\ \ y=\lambda-x}}\left( \bigoplus_{t \in \im x} M \setminus E_{x<t} / E_{x>t} \right)\ast \left( \bigoplus_{t \in \im y} N \setminus E_{y<t} / E_{y>t} \right).
    \end{equation*}
    Pick $\lambda\in\mathbb{R}^E$ with $\lambda(e) = -1$ for each $e \neq \epsilon$, and $\lambda(\epsilon)=0$.
    For each $x \in (M \wedge N)_\lambda$, let $y=\lambda-x$;
    we observe that this then implies $x \in \Trop(M)$ and $y \in \Trop(N)$.
    By \cite[Lemma 2.3]{ArdilaEurPenaguiao2023}, we have $x,y \in \{-1,0\}^E$, with $x(\epsilon)=y(\epsilon)=0$. Note that, given $E_1(x) := x^{-1}(\{0\})$ (i.e., we consider the indices where the coordinates of $x$ are zero) and $E_2(x) := y^{-1}(\{0\})$ for each $x \in \{-1,0\}^E$ with $x(\epsilon)=0$, we have
    \begin{align*}
        \left( \bigoplus_{t \in \im x} M \setminus E_{x<t} / E_{x>t} \right) &= M \setminus E_{x<0} / E_{x>0}   \oplus  M \setminus E_{x<-1} / E_{x>-1}  = M \setminus (E_2(x) \setminus \epsilon)  \oplus  M / E_1(x) \\
        \left( \bigoplus_{t \in \im y} N \setminus E_{y<t} / E_{y>t} \right) &= N \setminus E_{y<0} / E_{y>0}   \oplus  N \setminus E_{y<-1} / E_{y>-1}  = N \setminus (E_1(x) \setminus \epsilon)  \oplus  N / E_2(x).
    \end{align*}
    Using the bijection
    \begin{align*}
        \left\{x \in \{-1,0\}^E: x(\epsilon) =0\right\} &\longrightarrow \left\{(E_1,E_2)\in 2^E\times2^E: E_1 \cup E_2 = E, \ E_1 \cap E_2 = \{\epsilon\} \right\}\\
        x &\longmapsto \bigl(E_1(x),E_2(x)\bigr),
    \end{align*}
    the above equations combine to give
    \begin{align*}
        M\ast N &= \sum_{\substack{x \in \{-1,0\}^E\\ \ x(\epsilon)=0}}\Big( M \setminus (E_2(x) \setminus \epsilon)  \oplus  M / E_1(x) \Big)\ast \Big(N \setminus (E_1(x) \setminus \epsilon)  \oplus  N / E_2(x) \Big)\\
        &= \sum_{\substack{E_1,E_2\subseteq E\\ \ E_1 \cup E_2 = E\\E_1 \cap E_2 = \{\epsilon\} }}\Big( M \setminus (E_2 \setminus \epsilon)  \oplus  M / E_1 \Big)\ast \Big(N \setminus (E_1 \setminus \epsilon)  \oplus  N / E_2 \Big).
    \end{align*}
    The result now follows from applying \Cref{lem:cases_untangling}.
\end{proof}

\begin{remark}\label{rem:notinfinity}
    Since some of the summation terms given in \Cref{thm:main} can be $\infty$, it is not immediately clear if \Cref{thm:main} always provides a finite answer.
    That it does stems from the observation that, if we have some $E_1,E_2$ with $E_1 \cup E_2 = E$, $E_1 \cap E_2 = \{\epsilon\}$ and $(M/E_1) * (N \setminus E_1) = \infty$,
    then the flip product $(M\setminus E_2) * (N / E_2)$ is equal to 0.
    To see this, first observe that, as $(M/E_1) * (N \setminus E_1) = \infty$,
    it follows from \Cref{prop:intersection+number} that
    \begin{equation}\label{eq:notinfinity}
        r(M/E_1) + r(N \setminus E_1) > |E| - |E_1| + 1.
    \end{equation}
    Recall from minor rank formulas that $r(M/X) = r_M(E) - r_M(X)$ and $r(N \setminus X) = r_N(E \setminus X)$.
    Substituting these into \Cref{rem:notinfinity}, we get
    \begin{align*}
        |E| - |E_1| + 1 &< r_M(E) - r_M(E_1) + r_N(E \setminus E_1) \\
        \Longleftrightarrow \quad r_M(E_1) &< r_N(E \setminus E_1) + |E_1| + \underbrace{r_M(E) - |E| - 1}_{-r_N(E)}  \\
        &= r_N(E \setminus E_1) + |E_1| - r_N(E) \\
        &= r_{N^*}(E_1)
    \end{align*}
    where the final step follows by the dual rank formula.
    As removing an element can decrease the rank by at most one, we have $r_M(E_1 -\epsilon) \leq r_{N^*}(E_1 -\epsilon)$.
    Using this with the dual rank function for $N^*$, we observe that
    \begin{align*}
        r(M\setminus E_2) + r(N/E_2) &= r_M(E \setminus E_2) + r_N(E) - r_N(E_2) \\
        &= r_M(E_1 \setminus\epsilon) + r_N(E) - r_N(E \setminus (E_1 \setminus\epsilon)) \\
        &= r_M(E_1 \setminus \epsilon) - r_{N^*}(E_1 \setminus \epsilon) + |E_1 \setminus \epsilon| \\
        &\leq |E| - |E_2| \\
        &< |E| - |E_2| + 1 \, .
    \end{align*}
    With this, \Cref{prop:intersection+number} implies that $(M\setminus E_2) * (N / E_2) = 0$.
\end{remark}

We can drastically reduce the number of pairs $E_1,E_2$ required for the summation given in \Cref{thm:main} using \Cref{lem:coloop case}.
Doing so gives us the following corollary to \Cref{thm:main}, which more closely matches the main result of Capco et al.~\cite[Theorem 4.7]{CapcoGalletEtAl2018}.

\begin{corollary}\label{cor:main}
    Let $M, N$ be loopless matroids on $E$ such that $r(M) + r(N) = |E| + 1$, and let $e \in E$.
    Then the following equality holds:
    \begin{equation*}
        M \ast N := (M / e) \ast (N \setminus  e) + (M \setminus e) \ast (N / e) + \sum_{E_1,E_2} \Big( (M / E_1) \ast (N \setminus  E_1) \Big) \Big( (M \setminus E_2) \ast (N /E_2) \Big),
    \end{equation*}
    where each pair $E_1,E_2 \subset E$ satisfies the following conditions:
    \begin{enumerate}
        \item $E_1 \cup E_2 = E$ and $E_1 \cap E_2 = \{ e \}$;
        \item $|E_1| \geq 2$ and $|E_2| \geq 2$;
        \item $r (M / E_1) + r (N \setminus  E_1) = |E|-|E_1| +1$ and $r (M \setminus E_2) + r (N /  E_2) = |E| - |E_2| + 1$.
    \end{enumerate}
\end{corollary}

\section{Counting symmetric and periodic graph realisations}
\label{applications}

The original motivation for the flip product in \cite{ClarkeDewarEtAl2025} was to count the number of edge-length equivalent realisations for generic rigid frameworks in the plane.
In this section we extend this to counting realisations for frameworks with rotational symmetry or periodicity.
We additionally provide streamlined variants of \Cref{thm:main} for these cases,
since here we can store each matroid as a smaller object --- specifically, a $\Gamma$-gain multigraph (defined in \Cref{subsec:gain}).

\subsection{Section preliminaries}

Before we discuss applications, we first cover the following preliminaries on contractions, deletions and gain graphs and respective matroids.

\subsubsection{Contractions and deletions for representable matroids}

For a chosen characteristic 0 field $\mathbb{F}$, let $X$ be a matrix over $\mathbb{F}$ of size $|E| \times n$ and let $M$ be the matroid on $E$ given by $X$.
It is easy to see how to create a matrix from $X$ that generates the matroid $M \setminus \epsilon$ for any $\epsilon \in E$:
simply delete row~$\epsilon$ from~$X$.
The same operation for contracting a non-loop edge  $\epsilon \in E$ is a little trickier.
First, we choose a column $v \in [n]$ for which $X(\epsilon,v) \neq 0$, which is guaranteed to exist since $\epsilon$ is not a loop.
We now construct the $|E| \times n$ matrix $\tilde{X}$ formed by adding column $v$ scaled by $-X(\epsilon,u)/X(\epsilon,v)$ to column $u$ to cancel out the $X(\epsilon,u)$ entry.
With this,
the matrix $\tilde{X}$ generates the same matroid~$M$,
and the only non-zero entry for the $\epsilon$-th row is $X(\epsilon,v)$ in the $v$ column.
Now fix $Y$ to be the matrix formed by deleting row $\epsilon$ and column $v$ from~$\tilde{X}$.
With this, we see that the matroid for $Y$ is $M/\epsilon$.

\subsubsection{Gain multigraphs}\label{subsec:gain}

We fix throughout this section that $G=([n],E,f)$ is a (finite) multigraph;
i.e., $[n]$ is a vertex set, $E$ is a set of labels for the edges and $f$ maps a label to the vertices of that edge.
We define $\vec{E}$ to be the set of all possible ordered triples $(e,v,w)$, where $e \in E$ and $v,w \in [n]$ are the ends of $e$ (note that $\vec{E}$ contains both $(e,v,w)$ and $(e,w,v)$ when $v \neq w$).
Each element $(e,v,w) \in \vec{E}$ is said to be a \emph{directed edge} with \emph{source} $v$ and \emph{sink} $w$.

Now let $(\Gamma,+)$ be an abelian group with identity $0$.
A \emph{$\Gamma$-gain map} is a map $\phi \colon \vec{E} \rightarrow \Gamma$ where $\phi(e,v,w) = -\phi(e,w,v)$ for all $(e,v,w) \in \vec{E}$ where $v \neq w$.
We refer to the pair $(G,\phi)$ as a \emph{$\Gamma$-gain multigraph}.
We say that $(G,\phi)$ is a \emph{$\Gamma$-gain graph} if $\phi$ satisfies the additional conditions:
\begin{enumerate}
	\item for every distinct pair of edges $e,e'$ between vertices $v,w$ we have $\phi(e,v,w) \neq \phi(e',v,w)$, and
	\item $\phi(e,v,v) \neq 0$ for every loop $e\in E$.
\end{enumerate}
A \emph{gain graph/map/multigraph} is any $\Gamma$-gain graph/map/multigraph for some group $\Gamma$.

We always can form a gain (multi)graph by quotienting.
Specifically, if $H$ is a (multi)graph and $\Gamma$ is abelian group that acts on $H$ freely,
then we can form a $\Gamma$-gain (multi)graph by setting $G=H/\Gamma$ and using the gain map to `record' the original structure of $H$.
This process can be reversed to `lift' any $\Gamma$-gain (multi)graph to a unique (multi)graph known as its \emph{cover}.
For more detail on covers, see \cite[Section 2.1]{GrossTucker2001}\footnote{Gross and Tucker use slightly different terminology to our own: gains are `voltages', gain (multi)graphs are `voltage graphs' and covers are `derived graphs'.}.

One method for generating a new $\Gamma$-gain graph $(G,\phi')$ with the same cover as a $\Gamma$-gain graph $(G,\phi)$ is by using a \emph{switching operation}:
given a vertex $v$ and a group element $\gamma \in \Gamma$,
we define $\phi'$ to be the $\Gamma$-gain map where 
\begin{equation*}
    \phi'(e,x,y) =
    \begin{cases}
        \phi(e,x,y) &\text{if $x,y \neq v$ or $x=y=v$},\\
        \phi(e,v,y) + \gamma &\text{if $x=v$ and $y \neq v$}\\
        \phi(e,x,v) - \gamma &\text{if $x\neq v$ and $y = v$}.
    \end{cases}
\end{equation*}
We say that any two gain graphs are \emph{gain-equivalent} if there exists a sequence of switching operations that take one gain graph to the other.

One important concept we require is the following.
Let $W=(e_1,\ldots,e_t)$ be a walk in $G$.
We define $\phi(W):= \sum_{i=1}^t \phi(e_i)$.
A cycle $C$ of $G$ is called \emph{balanced} if $\phi(C)=0$.
A set $F\subseteq E$ is called \emph{balanced} if every cycle in $F$ is balanced.

\begin{remark}
    In this paper we are only concerned with abelian gain graphs: in particular, those coming from either the cycle group with $k$ elements $\mathbb{Z}_k$, or the rank $k$ commutative free group~$\mathbb{Z}^k$.
    All of the above concepts can, however, be extended to non-abelian groups.
    More on the general theory of non-abelian gain graphs can be found in \cite[Section 2]{GrossTucker2001}.
\end{remark}

\subsubsection{\texorpdfstring{$\ZZ_k$}{Zk}-gain graphic matroid}

Let $(G,\phi)$ be a $\ZZ_k$-gain multigraph with an arbitrary total ordering on its vertices.
Fix $I(G,\phi)$ to be the $|E|\times n$ matrix where any edge $e$ connecting vertices $v < w$ with gain $\gamma$ corresponds to the row
\begin{equation*}
    \setcounter{MaxMatrixCols}{12}
    \begin{bNiceMatrix}[first-row,first-col,code-for-first-row=\scriptstyle,code-for-first-col=\scriptstyle]
          &   &       &   & v &   &       &   & w                             &   &       &  \\
        e & 0 & \cdots & 0 & 1 & 0 & \cdots & 0 & -e^{2\pi\mathfrak{i}\gamma/k} & 0 & \cdots & 0
    \end{bNiceMatrix},
\end{equation*}
and any loop $\ell$ at vertex $v$ with gain $\gamma$ corresponds to the row
\begin{equation*}
    \begin{bNiceMatrix}[first-row,first-col,code-for-first-row=\scriptstyle,code-for-first-col=\scriptstyle]
          &   &       &   & v                        &   &       &   \\
        e & 0 & \cdots & 0 & 1-e^{2\pi\mathfrak{i}\gamma/k} & 0 & \cdots & 0
    \end{bNiceMatrix}.
\end{equation*}
We now fix $M(G,\phi)$ to be the row matroid formed from the matrix $I(G,\phi)$.

To get a better understanding of $M(G,\phi)$,
let us investigate the rank function of the matroid.
Choose any subset $F \subset E$ and let $(G[F],\phi|_F)$ be the $\ZZ_k$-gain graph formed by restricting to the multigraph $G[F]=(V[F],F)$ with induced vertex set
\begin{equation*}
    V[F] := \{ v \in [n] : v \text{ source or sink for some } e \in F \}.
\end{equation*}
We now see that
\begin{equation*}
    r(F) = \rank I(G[F],\phi|_F)
    = \bigl| V[F] \bigr| - \# \text{ connected balanced components of $G[F]$}.
\end{equation*}
It is now easy to check that gain equivalent $\mathbb{Z}_k$-gain graphs generate the same matroid.

\subsubsection{\texorpdfstring{$\mathbb{Z}^k$}{Zk}-gain graphic matroid}

Let $(G,\phi)$ be a $\mathbb{Z}^k$-gain multigraph with an arbitrary total ordering on its vertices.
Fix $I(G,\phi)$ to be the $|E|\times (n+k)$ matrix where any edge $e$ connecting vertices $v < w$ with gain $\gamma$ corresponds to the row
\begin{align*}
    \begin{bNiceMatrix}[first-row,first-col,code-for-first-row=\scriptstyle,code-for-first-col=\scriptstyle]
          &   &       &   & v &   &       &   & w  &   &       &   & \mathbb{Z}^k \\
        e & 0 & \cdots & 0 & 1 & 0 & \cdots & 0 & -1 & 0 & \cdots & 0 & \gamma^{\top}
    \end{bNiceMatrix},
\end{align*}
and any loop $\ell$ at vertex $v$ with gain $\gamma$ corresponds to the row
\begin{align*}
    \begin{bNiceMatrix}[first-row,first-col,code-for-first-row=\scriptstyle,code-for-first-col=\scriptstyle]
          &   &       &   & v &   &       &   & \mathbb{Z}^k \\
        e & 0 & \cdots & 0 & 0 & 0 & \cdots & 0 & \gamma^{\top}
    \end{bNiceMatrix}.
\end{align*}
We now fix $M(G,\phi)$ to be the row matroid formed from the matrix $I(G,\phi)$.

The rank function for $M(G,\phi)$ can be described as follows.
We define the \emph{lattice rank} of a connected $\mathbb{Z}^k$-gain multigraph $(G,\phi)$ (here denoted $\rank (G,\phi)$) to be the rank of the lattice
\begin{equation*}
    \langle \phi(C) \colon C \text{ is a directed cycle of } G \rangle.
\end{equation*}
Now let $F \subset E$ be an edge set and let $F_1,\ldots,F_n$ be the edge sets of the connected components of~$G[F]$.
With this, we have the rank function
\begin{equation*}
    r(F) = \rank I(G[F],\phi|_F)
    = \sum_{i=1}^n  |V[F_i]|-1 + \rank \big( G[F_i], \phi|_{F_i} \big).
\end{equation*}
It is now easy to check that gain-equivalent $\mathbb{Z}^k$-gain graphs generate the same matroid.

\subsection{Rotationally symmetric frameworks}

In lieu of a proper definition of rotationally-symmetric framework rigidity,
we instead opt to describe all properties using gain graphs and fibre dimensions.
For more background on rotational-symmetric frameworks, see \cite{JordanKaszanitzkyTanigawa2016}.

Let $(G,\phi)$ be a $\mathbb{Z}_k$-gain graph.
Given the rotation matrix
\begin{equation*}
    R_k :=
    \begin{bmatrix}
        \cos (2\pi/k) & -\sin (2\pi/k) \\
        \sin (2\pi/k) & \cos (2\pi/k)
    \end{bmatrix},
\end{equation*}
we define the \emph{rotational-symmetric measurement map} for~$(G,\phi)$:
\begin{equation*}
    m_{G,\phi} \colon (\mathbb{C}^2)^n  \longrightarrow \mathbb{C}^E, ~ p = \big(p(v) \big)_{v \in [n]} \longmapsto \left( \left\| p(v) -   R_k^{\phi(e)} p(w) \right\|^2  \right)_{(e,v,w)\in E},
\end{equation*}
where $\|\cdot\|^2$ is the quadratic form on $\mathbb{C}^2$ given by $(x,y) \mapsto x^2 + y^2$.

It is immediate that, given $A \in O(2,\mathbb{C})$,
the rotational-symmetric measurement map $m_{G,\phi}$ is invariant under the action $p \mapsto (Ap(v))_{v \in [n]}$.
This motivates the following definition for rigidity.

\begin{definition}
    A $\mathbb{Z}_k$-gain graph $(G,\phi)$ is \emph{$k$-fold rotation symmetric rigid} if $m_{G,\phi}^{-1}\bigl(m_{G,\phi}(p)\bigr)$ is 1-dimensional for generic $p \in (\mathbb{C}^2)^n$.
    If in addition the map $m_{G,\phi}$ is dominant, then $(G,\phi)$ is \emph{minimally $k$-fold rotation symmetric rigid};
    equivalently, $(G,\phi)$ is minimally $k$-fold rotation symmetric rigid if it is $k$-fold rotation symmetric rigid but is not if an edge is removed.
\end{definition}

Fortunately, minimal $k$-fold rotation symmetric rigidity (and hence also $k$-fold rotation symmetric rigidity) has an exact combinatorial characterisation.

\begin{theorem}[Malestein and Theran \cite{MalesteinTheran2015}]\label{thm:rotationalrigidity}
    Let $(G,\phi)$ be a $\ZZ_k$-gain graph.
    Then $(G,\phi)$ is minimally $k$-fold rotation symmetric rigid if and only if $|E|=2n-1$ and one of the following holds for every non-empty subset $F \subset E$:
    \begin{enumerate}
        \item $(G[F],\phi|_F)$ is balanced and $|F| \leq 2|V[F]|-3$; or
        \item $(G[F],\phi|_F)$ is not balanced and $|F| \leq 2|V[F]|-1$.
    \end{enumerate}
\end{theorem}

Given a minimally $k$-fold rotation symmetric rigid $\mathbb{Z}_k$-gain graph $(G,\phi)$,
we define $\ccountsym{G,\phi}{\phi}$ to be the cardinality of a generic fibre of $m_{G,\phi}$ modulo the action of $O(2,\mathbb{C})$ described above.
In the language of rigidity theory, $\ccountsym{G,\phi}{k}$ is counting the number of edge-length equivalent forced-symmetric frameworks one would expect for a generic choice of edge lengths that respect symmetry.
Using techniques from \Cref{hadamard_products},
we can now characterise this value:

\rotation*

\begin{proof}
    Set $\theta_k = 2\pi /k$ and choose any $p = \big( x(v) , y(v) \big)_{v \in [n]} \in (\mathbb{C}^2)^n$.
    We first observe the following equality that holds for any edge $(e,v,w)$ with gain $\gamma = \phi(e)$:
    \begin{align*}
    	&\Big(x(v) -  \cos (\gamma \theta_k) x(w) + \sin (\gamma \theta_k) y(w) \Big)^2 + \Big(
    	y(v) - \sin (\gamma \theta_k) x(w) - \cos (\gamma \theta_k) y(w)  \Big)^2 \\
    	= \quad &\left((x(v) + \ii y(v)) -  \ee^{\ii \gamma\theta_k}(x(w) + \ii y(w))\right)\left( (x(v) - \ii y(v)) - \ee^{-\ii \gamma\theta_k}(x(w) - \ii y(w)) \right)
    \end{align*}
    If we now apply the linear transformation $x(u) + \ii y(u) \mapsto \tilde x(u)$ and $x(u) - \ii y(u) \mapsto \tilde y(u)$ to the left hand side for each vertex $u \in [n]$,
    we see that we obtain the equation
    \begin{equation*}
    	m_{G,\phi}(p)_{(e,v,w)}=
        \left( \tilde x(v) -  \ee^{\ii \gamma\theta_k}\tilde x(w) \right) \left( \tilde y(v)  - \ee^{-\ii \gamma\theta_k} \tilde y(w) \right).
    \end{equation*}
    With this reformulation,
    we see that for generic $\lambda \in \mathbb{C}^E$ we have
    \begin{equation*}
        2 \ccountsym{G,\phi}{k} = \# \left( g_{X,Y}^{-1}(\lambda)/\!\!\sim_g \right),
    \end{equation*}
    where $X = I(G,\phi)$, $Y=I(G,-\phi)$ (with $-\phi$ being the gain map that takes value $-\phi(e)$ at each edge) and $g_{X,Y}$ is the map introduced in \Cref{subsec:gandh}.
    By comparing each matroid's rank function, we see that $M(G,-\phi) = M(G,\phi)$.
    Hence, $\ccountsym{G,\phi}{k} = \frac{1}{2} M(G,\phi) \ast M(G,\phi)$ by \Cref{thm:gmap}.
\end{proof}

\subsubsection{Deletion-contraction computation of realisation numbers}\label{subsubsec:rotationformula}

We now want to describe a variant of \Cref{thm:main} for $\ZZ_k$-gain matroids that is solely described by $\ZZ_k$-gain graphs.
To do so,
we need to understand how deletions and contractions of a gained edge~$(\epsilon,v,w)$ change a given $\ZZ_k$-gain matroid $M(G,\phi)$.

\paragraph{How to delete $\epsilon$:}
\begin{enumerate}
    \item Fix $\phi'$ to be the gain map for $G\setminus \epsilon$ formed from $\phi$ by restricting to $\{(e,v,w) \in \vec{E} : e \neq \epsilon\}$.
    \item It is now easy to see that $M(G\setminus \epsilon,\phi') = M(G,\phi)\setminus \epsilon$.
\end{enumerate}

\paragraph{How to contract $\epsilon$ if $w \neq v$:}
(Here we assume our contraction leaves the vertex $v$ and removes the vertex $w$.)
\begin{enumerate}
    \item By applying switching operations,
    form a gain-equivalent $\ZZ_k$-gain multigraph $(G,\phi')$ where $\phi'(\epsilon,v,w) = 0$.
    \item Fix $\phi''$ to be the gain map for $G/\epsilon$ formed from $\phi'$ by setting for any triple $(e,x,y)$ for $G/\epsilon$ stemming from $(e,x',y') \in \vec{E}$ the gain
    \begin{equation*}
        \phi''(e,x,y) =
        \begin{cases}
            \phi'(e,x',y') &\text{if } w \notin \{x',y'\}, \\
            \phi'(e,x',v)  &\text{if } y' = w, ~ x' \neq v, \\
            \phi'(e,v,y')  &\text{if } x' = w, ~ y' \neq v, \\
            \phi'(e,v,w)   &\text{if } \{x',y'\} = \{v,w\}.
        \end{cases}
    \end{equation*}
    \item We now have $M(G/\epsilon,\phi'') = M(G,\phi')/ \epsilon$.
\end{enumerate}

\paragraph{How to contract $\epsilon $ if $w = v$ and $\phi(\epsilon,v,v) \neq 0$:}
Here we assume that the loops at $v$ are $\epsilon,\ell_1,\ldots,\ell_t$
\begin{enumerate}
    \item Fix $G'$ to be the graph formed from $G$ by deleting the vertex $v$,
    changing each edge from $w \neq v$ to $v$ to be a loop at $w$,
    and adding the loops $\ell_1,\ldots,\ell_t$ at an arbitrary vertex $u \neq v$.
    If $G$ only has a single vertex, we rename the vertex $v$ to be $u$ and add the loops $\ell_1,\ldots,\ell_t$ to $u$.
    \item Fix $\phi'$ to be the gain map for $G'$ formed from $\phi$ by setting for any triple $(e,x,y)$ for $G'$ stemming from $(e,x',y') \in \vec{E}$ the gain
    \begin{equation*}
        \phi'(e,x,y) =
        \begin{cases}
            \phi(e,x',y') & \text{if } v \notin \{x',y'\}, \\
            1 & \text{if } x' = v \text{ and } y' \neq v, \text{ or vice versa}, \\
            0 & \text{if } e = \ell_i \text{ for some } i \in [t].
        \end{cases}
    \end{equation*}
    Note that the 0 gain on loops here is intentional.
    \item We now have $M(G',\phi') = M(G,\phi)/ \epsilon$.
\end{enumerate}
Contracting a loop with gain 0 is just the deletion of that loop.

\begin{example}
    Let $(G,\phi)$ be the $\mathbb{Z}_4$-gain graph depicted in \Cref{fig:symgain}.
    \begin{figure}[ht]
        \centering
        \begin{tikzpicture}[]
            \node[vertex,label={[labelsty]left:$v_1$}] (1) at (0,0) {};
            \node[vertex,label={[labelsty]right:$v_2$}] (2) at (0:2) {};
            \node[vertex,label={[labelsty,label distance=4pt]below:$v_3$}] (3) at (60:2) {};
            \draw[dedge] (1) to[bend left=20] node[above,labelsty] {1} (2);
            \draw[dedge] (2) to[bend left=20] node[below,labelsty] {0} (1);
            \draw[dedge] (2) to node[right,labelsty] {2} (3);
            \draw[dedge] (3) to node[left,labelsty] {3} (1);
            \draw[dedge] (3) to[out=60,in=120,loop,min distance=1cm] node[above,labelsty] {1} (3);
        \end{tikzpicture}
        \qquad
        \begin{tikzpicture}[]
            \foreach \w/\r [count=\kk] in {0/1,60/1.6,30/1.5}
            {
                \foreach \m [count=\mm] in {0,1,2,3}
                {
                    \node[vertex,label={[labelsty]{\w+\m*90}:$v_\kk^{(\mm)}$}] (v\kk\mm) at (\w+\m*90:\r) {};
                }
            }
            \foreach \s/\t/\g in {1/2/1,2/1/0,2/3/2,3/3/1,3/1/3}
            {
                \foreach \i [evaluate=\i as \k using {int(mod(\i+\g-1,4)+1)}] in {1,2,3,4}
                {
                    \draw[edge] (v\s\i)--(v\t\k);
                }
            }
        \end{tikzpicture}
        \caption{A $\ZZ_4$-gain graph (left) and a 4-fold symmetric realisation of the cover (right).}
        \label{fig:symgain}
    \end{figure}
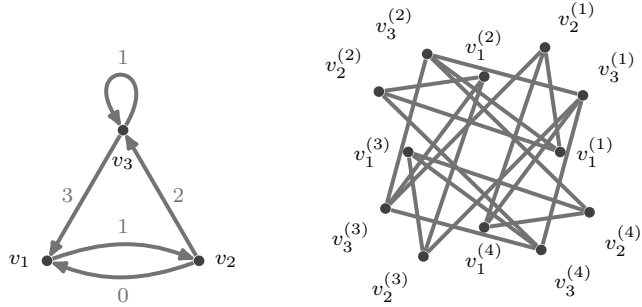
    
    Using \Cref{thm:rotation} we can compute now $\ccountsym{G,\phi}{4}=6$.
    For this graph we can indeed find for certain edge lengths six real realisations by solving the system of distance equations; see for example \Cref{fig:symrealisations}. Note however, that this latter method is computationally not feasible for larger graphs.

    \begin{figure}
        \centering
        \foreach \coord in {{(1.76881,0.),(0.,-1.76881),(-1.76881,0.),(0.,1.76881),(5.85538,4.15932),(4.15932,-5.85538),(-5.85538,-4.15932),(-4.15932,5.85538),(1.08055,-3.21441),(-3.21441,-1.08055),(-1.08055,3.21441),(3.21441,1.08055)},{(1.86128,0.),(0.,-1.86128),(-1.86128,0.),(0.,1.86128),(-2.38468,-3.99647),(-3.99647,2.38468),(2.38468,3.99647),(3.99647,-2.38468),(1.6467,-2.96452),(-2.96452,-1.6467),(-1.6467,2.96452),(2.96452,1.6467)},{(2.68335,0.),(0.,-2.68335),(-2.68335,0.),(0.,2.68335),(-2.14746,-3.26547),(-3.26547,2.14746),(2.14746,3.26547),(3.26547,-2.14746),(-3.10643,-1.36018),(-1.36018,3.10643),(3.10643,1.36018),(1.36018,-3.10643)},{(4.63083,0.),(0.,-4.63083),(-4.63083,0.),(0.,4.63083),(6.24959,5.60175),(5.60175,-6.24959),(-6.24959,-5.60175),(-5.60175,6.24959),(-3.30723,0.749824),(0.749824,3.30723),(3.30723,-0.749824),(-0.749824,-3.30723)},{(8.14689,0.),(0.,-8.14689),(-8.14689,0.),(0.,8.14689),(2.88893,2.52069),(2.52069,-2.88893),(-2.88893,-2.52069),(-2.52069,2.88893),(1.16838,3.18353),(3.18353,-1.16838),(-1.16838,-3.18353),(-3.18353,1.16838)},{(8.38989,0.),(0.,-8.38989),(-8.38989,0.),(0.,8.38989),(3.44073,3.08316),(3.08316,-3.44073),(-3.44073,-3.08316),(-3.08316,3.44073),(-0.636965,3.33081),(3.33081,0.636965),(0.636965,-3.33081),(-3.33081,-0.636965)}}
        {
            \begin{tikzpicture}[scale=0.25,yscale=-1]
                \clip (-9,-9) rectangle (9,9);
                \foreach \pp [count=\ii,evaluate=\ii as \mm using {int(mod(\ii-1,4)+1)},evaluate=\ii as \kk using {int(int((\ii-1)/4)+1)}] in \coord
                {
                    \node[vertex] (v\kk\mm) at \pp {};
                }
                \foreach \s/\t/\g in {1/2/1,2/1/0,2/3/2,3/3/1,3/1/3}
                {
                    \foreach \i [evaluate=\i as \k using {int(mod(\i+\g-1,4)+1)}] in {1,2,3,4}
                    {
                        \draw[edge] (v\s\i)--(v\t\k);
                    }
                }
            \end{tikzpicture}%
        }
        \caption{Six real realisations of the graph from \Cref{fig:symgain}.}
        \label{fig:symrealisations}
    \end{figure}
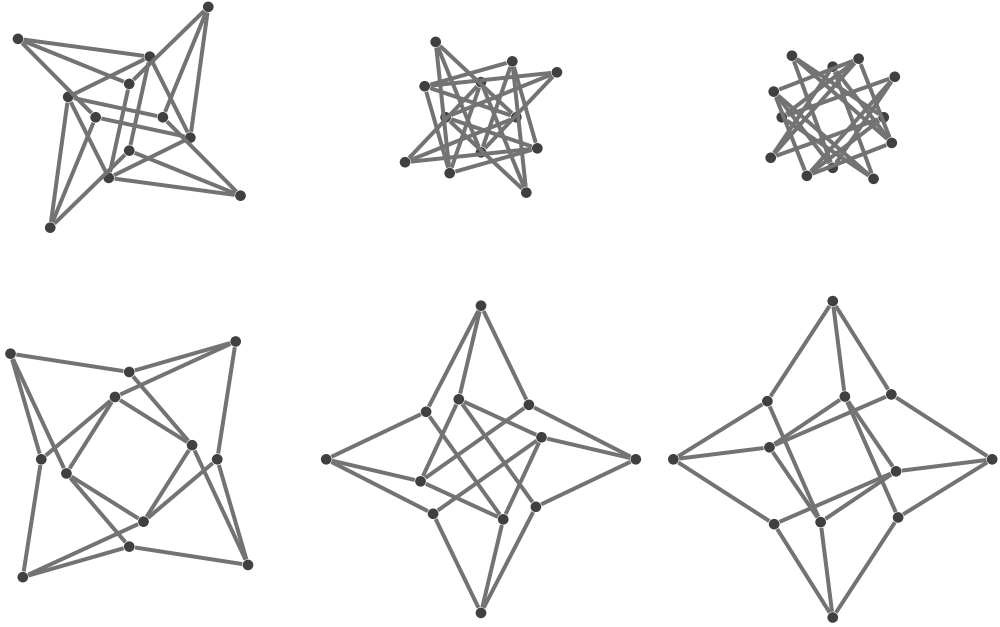
\end{example}

\subsection{Periodic frameworks with flexible lattices}

Again, in lieu of a proper definition for periodic framework rigidity,
we instead opt to describe all properties using gain graphs and fibre dimensions.
For more on the general theory surrounding periodic frameworks, see \cite{BorceaStreinu2010}.

Let $(G,\phi)$ be a $\mathbb{Z}^k$-gain graph.
We define the following map, known here as the \emph{periodic measurement map} for $(G,\phi)$:
\begin{equation*}
    m_{G,\phi} \colon (\mathbb{C}^2)^n \times \mathbb{C}^{2 \times k}  \longrightarrow \mathbb{C}^E, ~ (p, L) = \left(\big(p(v) \big)_{v \in [n]} , L \right) \longmapsto \left( \left\| p(v) - p(w) + L\phi(e) \right\|^2  \right)_{(e,v,w)\in E},
\end{equation*}
where $\|\cdot\|^2$ is the quadratic form on $\mathbb{C}^2$ given by $(x,y) \mapsto x^2 + y^2$.

It is immediate that, given $A \in O(2,\mathbb{C})$ and $q \in \mathbb{C}^2$,
the periodic measurement map~$m_{G,\phi}$ is invariant under the action $p \mapsto (Ap(v) +q)_{v \in [n]}$.
This motivates the following definition for rigidity.

\begin{definition}
    A $\mathbb{Z}^k$-gain graph $(G,\phi)$ is \emph{periodically rigid} if $m_{G,\phi}^{-1}(m_{G,\phi}(p,L))$ is 3-dimensional for generic $p \in (\mathbb{C}^2)^n$ and $L \in \mathbb{C}^{2 \times k}$.
    If in addition the map $m_{G,\phi}$ is dominant, then $(G,\phi)$ is \emph{minimally periodically rigid};
    equivalently, $(G,\phi)$ is minimally periodically rigid if it is periodically rigid but is not if an edge is removed.
\end{definition}

Here we also have an exact combinatorial characterisation.

\begin{theorem}[Malestein and Theran \cite{MalesteinTheran2013}]\label{thm:periodicrigidity}
    Let $(G,\phi)$ be a $\ZZ^k$-gain graph for $k \in \{1,2\}$.
    Then $(G,\phi)$ is minimally periodically rigid if and only if $|E|=2n-3 +2k$ and the following holds for every non-empty subset $F \subset E$ that induces a connected subgraph:
    \begin{equation*}
        |F| \leq
        \begin{cases}
            2|V[F]|-3 &\text{if } (G[F],\phi|_F) \text{ is balanced},\\
            2|V[F]|-1 &\text{if } \rank(G[F],\phi|_F) = 1,\\
            2|V[F]|+1 &\text{if } \rank(G[F],\phi|_F) = 2.
        \end{cases}
    \end{equation*}
\end{theorem}

\begin{remark}
    Malestein and Theran's main result in \cite{MalesteinTheran2013} is stated for the $k=2$ case. To get the $k=1$ case, we convert each gain $n \in \mathbb{Z}$ into the lattice point $(n,0) \in \mathbb{Z}^2$,
    then add two loops with gain $(1,0)$ and $(1,1)$.
    With this, the original $\mathbb{Z}$-gain graph is minimally periodically rigid if and only if the new $\mathbb{Z}^2$-gain graph is minimally periodically rigid.
\end{remark}

Given a minimally periodically rigid $\mathbb{Z}^k$-gain graph $(G,\phi)$,
we define $\ccountper{G,\phi}{k}$ to be the cardinality of a generic fibre of $m_{G,\phi}$ modulo the action of $O(2,\mathbb{C}) \ltimes \mathbb{C}^2$ described above.
In the language of rigidity theory, $\ccountper{G,\phi}{k}$ is counting the number of edge-length equivalent periodic frameworks (allowing for lattice deformation) one would expect for a generic choice of edge lengths that respect symmetry.
Using techniques from \Cref{hadamard_products},
we can now characterise this value:

\periodic*

\begin{proof}
    Here we observe that $m_{G,\phi}$ is exactly the map $h_X$ described in \Cref{subsec:gandh} with $X=I(G,\phi)$.
    The result now follows from \Cref{thm:hmap}.
\end{proof}

\subsubsection{Deletion-contraction formula for periodic realisation numbers}\label{subsubsec:periodicformula}

As like before, we want to describe a variant of \Cref{thm:main} for $\mathbb{Z}^k$-gain matroids that is solely described by $\mathbb{Z}^k$-gain graphs.
To do so,
we need to understand how deletions and contractions of a gained edge $(\epsilon,v,w)$ affect a given $\mathbb{Z}^k$-gain matroid $M(G,\phi)$.

\paragraph{How to delete $\epsilon$:}
\begin{enumerate}
    \item Fix $\phi'$ to be the gain map for $G\setminus \epsilon$ formed from $\phi$ by restricting to $\{(e,v,w) \in \vec{E} : e \neq \epsilon\}$.
    \item It is now easy to see that $M(G\setminus \epsilon,\phi') = M(G,\phi)\setminus \epsilon$.
\end{enumerate}

\paragraph{How to contract $\epsilon$ if $w \neq v$:}
(Here we assume our contraction leaves the vertex $v$ and removes the vertex $w$.)
\begin{enumerate}
    \item By applying switching operations,
    form a gain-equivalent $\mathbb{Z}^k$-gain multigraph $(G,\phi')$ where $\phi'(\epsilon,v,w) = 0$.
    \item Fix $\phi''$ to be the gain map for $G/\epsilon$ formed from $\phi'$ by setting for any triple $(e,x,y)$ for $G/\epsilon$ stemming from $(e,x',y') \in \vec{E}$ the gain
    \begin{equation*}
        \phi''(e,x,y) =
        \begin{cases}
            \phi'(e,x',y') &\text{if } w \notin \{x',y'\}, \\
            \phi'(e,x',v)  &\text{if } y' = w, ~ x' \neq v, \\
            \phi'(e,v,y')  &\text{if } x' = w, ~ y' \neq v, \\
            \phi'(e,v,w)   &\text{if } \{x',y'\} = \{v,w\}.
        \end{cases}
    \end{equation*}
    \item We now have $M(G/\epsilon,\phi'') = M(G,\phi)/ \epsilon$.
\end{enumerate}

\paragraph{How to contract $\epsilon $ if $w = v$ and $\phi(\epsilon,v,v) \neq 0$:}
Here we set $\gamma:=\phi(\epsilon,v,v)$ and we assume that the loops at $v$ are $\epsilon,\ell_1,\ldots,\ell_t$.
\begin{enumerate}
    \item Fix $G'$ to be the graph formed from $G$ by deleting the loop $\epsilon$.
    \item Choose an index $j \in [k]$ so that $\gamma_j \neq 0$.
    \item Fix $\phi'$ to be the gain map for $G'$ formed from $\phi$ by setting
    \begin{equation*}
        \phi'(e,x,y) = \gamma_j \phi(e,x,y) - \gamma \odot \phi(e,x,y),
    \end{equation*}
    where $\odot$ represents the Hadamard product.
    \item We now have $M(G\setminus \epsilon,\phi') = M(G,\phi)/ \epsilon$.
\end{enumerate}

\section{Matroid applications of the flip product}
\label{matroid_applications}

In this section we outline some matroidal applications for the flip product.

\subsection{Beta invariant}

Let $M$ be a matroid with ground set $E$.
The \emph{beta invariant of $M$} is the non-negative integer
\begin{equation*}
    \beta(M) := (-1)^{r(M)} \sum_{X \subseteq E} (-1)^{|X|} r_M(X).
\end{equation*}
In \cite{ArdilaEurPenaguiao2023},
Ardila-Mantilla, Eur and Penaguiao described the beta invariant in terms of tropical objects:

\begin{theorem}[Ardila-Mantilla, Eur and Penaguiao \cite{ArdilaEurPenaguiao2023}]\label{thm:aep}
    Let $M$ be a matroid with ground set~$E$.
    Then the beta invariant $\beta(M)$ of $M$ satisfies the following equality for any $\epsilon \in E$ that is neither a coloop nor loop:
    \begin{equation*}
        \beta(M) = \left\{ x \in \mathbb{R}^{E \setminus \epsilon} : (x,0) \in \Trop (M) \right\} \cdot (-\Trop \big((M/\epsilon)^*) \big).
    \end{equation*}
\end{theorem}

We now recharacterise this result in terms of flip products.

\begin{theorem}\label{thm:beta}
    Let $M$ be a matroid with ground set $E$, where $|E|\geq2$.
    If $\epsilon \in E$ is any element, then
    \begin{equation*}
        \beta(M) = (M \setminus \epsilon) \ast (M^* \setminus \epsilon).
    \end{equation*}
\end{theorem}

\begin{proof}
    Consider the case when $\epsilon\in E$ is a coloop of $M$ (and hence a loop of $M^*$). As $M$ is disconnected and $|E|\geq2$, we have $\beta(M)=0$.
    Since $\epsilon$ is a coloop of $M$ but not a coloop of $M^*$, we have
    \begin{equation*}
        r(M\setminus\epsilon) = r(M)-1, \qquad r(M^*\setminus\epsilon) = r(M^*)=|E|-r(M).
    \end{equation*}
    Hence,
    \begin{equation*}
        r(M\setminus\epsilon)+r(M^*\setminus\epsilon)=|E|-1 <  |E\setminus\epsilon|+1,
    \end{equation*}
    and thus $(M \setminus \epsilon) \ast (M^* \setminus \epsilon)=0$ by \Cref{prop:intersection+number}.
    If, instead, $\epsilon\in E$ is a loop of $M$, then it is a coloop of $M^*$ and a similar count shows $\beta(M)=0=(M \setminus \epsilon) \ast (M^* \setminus \epsilon)$.
    We are left with the case when $\epsilon\in E$ is neither a loop or a coloop.

    We first note that $(M/\epsilon)^* = (M^*
    \setminus \epsilon)$.
    It follows from \Cref{thm:aep} that
    \begin{equation}\label{eq:beta1}
        \beta(M) = \Trop (M) \cdot (-\Trop (M^*\setminus \epsilon) \times \mathbb{R}) \cdot \Trop(y_\epsilon - 1).
    \end{equation}
    Fix $N=(M^*\setminus\epsilon) \oplus U_{\{\epsilon\},1}$, where $U_{\{\epsilon\},1}$ is the uniform matroid on the set $\{\epsilon\}$ of rank $1$.
    Then
    \begin{equation*}
        \Trop (N) = \Trop (M^*\setminus \epsilon) \times \Trop(U_{\{\epsilon\},1}) = \Trop (M^*\setminus \epsilon) \times \mathbb{R}.
    \end{equation*}
    When combined with \cref{eq:beta1}, we have
    \begin{equation}\label{eq:beta2}
        \beta(M) = \Trop (M) \cdot (-\Trop (N))\cdot \Trop(y_\epsilon - 1) = M \ast N.
    \end{equation}
    Since $r(N) = r(M^*) + 1 = |E| - r(M) +1$,
    we have $r(M) + r(N) = |E|+1$.
    As $\epsilon$ is not a coloop of $M$ but it is a coloop of $N$,
    we can apply \Cref{lem:removebridge} to the $M \ast N$ term in \cref{eq:beta2} to obtain the desired result.
\end{proof}

\begin{remark}
    The statement of \Cref{thm:beta} at first looks a little odd, since it is a statement about a matroid gleamed from two of its minors.
    However, if we know the matroids~$M \setminus \epsilon$ and~$M/\epsilon$, we can reconstruct~$M$ uniquely.
    One method for seeing this is via bases: a set $B$ is a basis of $M$ if and only if it does not contain $\epsilon$ and is a basis of $M \setminus \epsilon$, or it does contain $\epsilon$ and~$B \setminus \epsilon$ is a basis of~$M/\epsilon$.
    Now, the knowledge of $M \setminus \epsilon$ $M / \epsilon$ is equivalent to the one of $M \setminus \epsilon$ and $M^{\ast} \setminus \epsilon$, since $(M/\epsilon)^* = M^* \setminus \epsilon$ and matroid duality is invertible.
    Hence, we can uniquely reconstruct $M$ from our original input.
\end{remark}

\begin{remark}
    An \emph{$n$-state linear discrete probability model} consists of a linear space $X \subset \mathbb{R}^n$ that intersects the probability simplex 
    $\Delta_n := \{p \in \mathbb{R}^n_{\geq 0} : \sum_{i=1}^n p(i) = 1\}$ and is not contained in any coordinate hyperplane.
    Given $u(i)$ is the fraction of samples in state $i$ from our sampled data,
    we determine the best-fitting probability distribution by maximising the log-likelihood function
    \begin{equation*}
        \log \mathcal{L}_u \colon X \cap \Delta_n \longrightarrow \mathbb{R}, ~ p \longmapsto \sum_{i=1}^n  u(i) \log p(i) .
    \end{equation*}
    Here we observe that domain of the derivative map $p \mapsto D_p(\log \mathcal{L}_u)$ can be extended to $Y \cap (\mathbb{C}\setminus \{0\})^n$,
    where $Y$ is the complex Zariski-closure of $X \cap \Delta_n$ (and hence an affine space).
    With this in mind,
    we define the \emph{maximum likelihood degree (ML-degree) of $X$} to be the number of points $p \in Y \cap (\mathbb{C}\setminus \{0\})^n$ for which the derivative $D_p(\log \mathcal{L}_u)$ is the zero map, where $u$ is an arbitrary generic point of $\Delta_p$.
    In some sense, the ML-degree of $X$ represents the `complexity' of working with the linear discrete probability model described by $X$.
    Given $Z \subset \mathbb{C}^{n+1}$ is the linear space formed from homogenising the affine space $Y$, it now follows from a result of Agostini et al.~\cite[Theorem 7.1]{AgostiniBrysiewiczEtAl2023} along with \Cref{thm:aep} and \Cref{thm:beta} that the ML-degree of $X$ is exactly
    \begin{equation*}
        \Big(M(Z) \setminus (n+1) \Big) \ast \Big(M(Z)^* \setminus (n+1)\Big).
    \end{equation*}
\end{remark}

\subsection{Coefficients of the characteristic polynomial and non-broken circuit bases}\label{matroid_application:characteristic_polynomial}

Let $M$ be a matroid with ground set $E$. Its \emph{characteristic polynomial} is defined as
\begin{equation*}
    p_M(\lambda) = \sum_{A\subseteq E}(-1)^{|A|}\lambda^{r_M(E)-r_M(A)}.
\end{equation*}
Notice that $p_M(1)=0$, so we can define the \emph{reduced characteristic polynomial} as the degree $r(M)-1$ polynomial
\begin{equation*}
    \overline{p}_M(\lambda) = p_M(\lambda)/(\lambda -1) = \sum_{i=0}^{r(M)-1}(-1)^{i}\mu_i\lambda^{r(M)-1-i}.
\end{equation*}
Translated into our setting, the content of \cite[Proposition 5.2]{HuhKatz2012} is that each coefficient $\mu_i$ can be expressed as flip products of a uniform matroid with a `rank-reduced' copy of $M$.
For $k \in \{0, \ldots, r(M)\}$, we set the \emph{truncation} $\mathrm{Trunc}_k(M)$ of $M$ to be the matroid with ground set $E$ and rank function $r_k(A)=\min\{r_M(A),k\}$.

\begin{theorem}[Huh and Katz \cite{HuhKatz2012}]
    For a simple matroid $M$,
    the coefficients of $\overline{p}_M(\lambda)$ are given by
    \begin{equation*}
        \mu_k = U_{E,|E|-k} \ast \mathrm{Trunc}_{k+1}(M).
    \end{equation*}
\end{theorem}

We obtain a particular case of this result by considering the constant coefficient $\mu_{r(M)-1}$, which has a special meaning explained below.

Let $M$ be a matroid with ground set $E$ and a total ordering $<$ on $E$.
A \emph{broken circuit} is a set $C \subset E$ formed from a circuit $C' \subseteq E$ by removing its smallest element with respect to~$<$.
A \emph{non-broken circuit basis} (nbc-basis) is any basis of $M$ that does not contain a broken circuit.
Although the set of nbc-bases of a matroid is dependent on the choice of ordering $<$,
the number of nbc-bases of a matroid is independent of this choice and, in fact, coincide with the number $\mu_{r(M)-1}$ (see \cite[Theorem 7.4.6]{Bjorner1992}).

\begin{theorem}[Huh and Katz \cite{HuhKatz2012}; see also \cite{ClarkeDewarEtAl2025}] \label{thm:nbc}
    Let $M$ be a matroid of rank $r$ with ground set~$[n]$.
    Then
    \begin{equation*}
        M \ast U_{n,n-r+1} = \nbc(M).
    \end{equation*}
\end{theorem}

\begin{remark}
    Every loopless matroid has at least one nbc-basis under any ordering.
    Hence, \Cref{thm:nbc} informs us that for every loopless matroid $M$ there exists a matroid $N$ that shares the same ground set such that $M \ast N$ is a positive integer.
\end{remark}

\begin{remark}
    It is easy to see from the definition that the number of nbc-bases of a uniform matroid~$U_{n,r}$ is $\binom{n-1}{r-1}$.
    This can now also be derived by picking generic subspaces $U,V \subseteq \CC^n$ such that $M(U) = U_{n,r}$ and $M(V) = U_{n,n-r+1}$ and applying \Cref{thm:mainalg} and \Cref{cor:mainalg} alongside \Cref{thm:nbc}:
    \begin{equation*}
        \binom{n-1}{r-1} = \# \left(f^{-1}_{U,V}(\lambda) /\!\!\sim \right) = U_{n,r} * U_{n,n-r+1} = \nbc(U_{n,r}) \, .
    \end{equation*}
\end{remark}

Clarke et al.~proved in \cite{ClarkeDewarEtAl2025} that nbc-bases provide an effective upper bound for flip products.

\begin{theorem}[Clarke et al.~{\cite[Lemma 5.8]{ClarkeDewarEtAl2025}}]\label{thm:nbcupperbound}
    Let $M,N$ be matroids with shared ground set $E$ where $r(M) + r(N) = |E| +1$.
    Then
    \begin{equation*}
        M \ast N \leq \min \{\nbc(M),\nbc(N)\}.
    \end{equation*}
\end{theorem}

\begin{remark}
    If a matroid $M$ has a loop, then the empty set is a broken circuit under any choice of ordering.
    This in turn implies that $\nbc(M)=0$.
    \Cref{thm:nbcupperbound} now implies $M \ast N = 0$ for any choice of matroid $N$ that shares a ground set with $M$,
    which agrees with \Cref{prop:loopbad}.
\end{remark}

\section{Flip products and matroid subdivision}
\label{weak_order}

Recall that the set of matroids on a fixed ground set $E$ form a poset under the \emph{weak order of matroids}, where $M_1 \preceq M_2$ if every independent set of $M_1$ is also an independent set of $M_2$.
If $M_1$ and $M_2$ are of the same rank, we can also characterise $M_1 \preceq M_2$ as every basis of $M_1$ is a basis of~$M_2$.

A straightforward corollary of \Cref{thm:nbc} is that if $M \preceq M'$ with $r(M) = r(M') = k$, the flip product $M * U_{n,n-k+1} \leq M' * U_{n,n-k+1}$ can only increase.
This follows as every nbc-basis of~$M$ must also be an nbc-basis of $M'$.
A natural question is whether we can replace the uniform matroid with an arbitrary matroid of the correct rank.

\begin{question} \label{q:flip+weak+order}
    Let $M$ and $N$ be matroids on $E$ such that $r(M) + r(N) = |E| + 1$.
    If $M'$ a matroid satisfying $M \preceq M'$ and $r(M) = r(M')$, do we have $M*N \leq M'*N$?
\end{question}

We give a partial answer to this question by demonstrating the flip product increases under a certain polyhedral condition in terms of matroid polytopes.

We begin with the following definition.

\begin{definition}\label{def: matroid subdivision}
    Given a set system $\cB \subseteq \binom{E}{k}$, we define the polytope
    \begin{equation*}
        P(\cB) = \conv(\chi_B \mid B \in \cB) \subseteq \RR^E \, , \quad \chi_B = \sum_{e\in B} \chi_e \, .
    \end{equation*}
    If $\cB = \cB(M)$ is the set of bases of a matroid $M$, then we call this the \emph{matroid polytope} of $M$, and denote it by $P(M) := P(\cB(M))$.

    Given a weight function $p \colon \cB(M) \rightarrow \RR$ on the bases of $M$, we let $\Delta_p(M)$ be the polyhedral subdivision of $P(M)$ whose pieces are of the form
    \begin{equation*}
        \Delta_p(M) = \{P(\cB_{p,w}) \mid w \in \RR^E\} \, , \quad \cB_{p,w} = \left\{ B' \in \cB \, \Big| \, \min_{B \in \binom{E}{k}} \bigg(p_B - \sum_{i \in B} w_i  \bigg) \text{ attained at } B' \right\} \, .
    \end{equation*}
    We say that $\Delta_p(M)$ is a \emph{matroid subdivision}\footnote{This definition implicitly requires the subdivision to be \emph{regular}, i.e., induced by a weight function. Some places in the literature do not require that the subdivision need be regular. We do insist this as only regular subdivisions have a correspondence with tropical linear spaces.} if it is a subdivision into matroid polytopes, i.e., each~$\cB_{p,w}$ is the set of bases of some matroid $M_{p,w}$.
\end{definition}

We are interested in matroid polytopes and matroid subdivisions for two reasons.
The first is that they give an alternative characterisation of tropical linear spaces, detailed in \cite{Speyer2008} and \cite[Section 4.4]{MaclaganSturmfels2015}.
Let $M$ be a matroid and $p \colon \binom{E}{k} \rightarrow \RR \cup \{\infty\}$ where $p(B) \neq \infty$ if and only if $B \in \cB(M)$.
We can consider $p$ as a weight function on $\cB(M)$ and so the subdivision $\Delta_p(M)$ is well-defined.
We say $p$ is a \emph{tropical Pl\"ucker vector} if $\Delta_p(M)$ is a matroid subdivision, or equivalently $\cB_{p,w}$ is the set of bases of a matroid $M_{p,w}$ for all $w \in \RR^n$.
We define its associated \emph{tropical linear space} to be
\begin{equation}\label{eq:troplinearspace}
    L_p = \left\{w \in \RR^n \mid M_{p,w} \text{ is a loopless matroid}\right\} .
\end{equation}
The set $L_p$ can be endowed with the structure of a balanced polyhedral complex where each maximal cell has multiplicity one.
We emphasise that Bergman fans fit into this regime: the Bergman fan $\Trop(M)$ is the tropical linear space $L_q$ where $q(B)$ equals zero if $B \in \cB(M)$ and infinity otherwise.

The second reason we are interested in matroid polytopes is they give a geometric intuition to the weak order of matroids.
It is straightforward from the definitions that $M \preceq M'$ if and only if $P(M) \subseteq P(M')$.
We investigate a stronger property, where there exists a matroid subdivision~$\Delta(M')$ of $P(M')$ that contains $P(M)$ as a cell.
Our main result of this section is that under this condition, the flip product is increasing.

\begin{proposition}\label{prop:flip+relaxation}
    Let $M$ and $N$ be matroids on $E$ such that $r(M) + r(N) = |E| + 1$.
    Let $M'$ be a matroid such that there exists a matroid subdivision $\Delta(M')$ of $P(M')$ containing $P(M)$ as a cell.
    Then $M*N \leq M'*N$.
\end{proposition}

We prove this by deforming $\Trop(M)$ into $\Trop(M')$ and tracking its intersection points with $-\Trop(N)$.
A similar proof technique was implemented in proving \cite[Lemma 5.8]{ClarkeDewarEtAl2025}, which is the special case where $M'$ is uniform: in this case, the matroid subdivision of $P(M')$ containing $P(M)$ is the corank subdivision described in \cite{JoswigSchroter2017}.
In fact, this special case was used in \cite{ClarkeDewarEtAl2025} (in combination with \Cref{thm:nbc}) to prove \Cref{thm:nbcupperbound}.

\begin{lemma}\label{lem:degrees}
    Let $\Sigma$ be a balanced polyhedral complex in $\RR^n$ of dimension $1 \leq s \leq n-1$ and let $L_q$ a tropical linear space with tropical Pl\"{u}cker vector $q\colon {\binom{[n]}{n-s}} \rightarrow \RR\cup \{\infty\}$.
    Let $M_q$ be the matroid where $B$ is a basis if and only if $q_B \neq \infty$.
    Then $\Sigma \cdot L_q = \Sigma \cdot \Trop(M_q)$.
\end{lemma}

\begin{proof}
    \cite[Theorem 4.4.5]{MaclaganSturmfels2015} states that the recession fan of $L_q$ (see \Cref{def:recessionfan}) is precisely $\Trop(M_q)$.
    Furthermore, the recession fan of $\Trop(M_q)$ is itself as it is already a fan.
    Thus $\rec(L_q)=\rec(\Trop(M_q))$.
    By \cite[Theorem 5.7]{AllermannHampeRau2016}, if $X$ and $Y$ are polyhedral complexes with codimension complementary to that of $\Sigma$, then $\rec(X) = \rec(Y)$ implies that $\Sigma \cdot X = \Sigma \cdot Y$. Thus $\Sigma \cdot L_q = \Sigma \cdot \Trop(M_q)$.
\end{proof}

\begin{proof}[Proof of \Cref{prop:flip+relaxation}]
    If $M$ contains a loop, by \Cref{prop:intersection+number} \ref{prop:intersection+numberitem2} we have $M*N = 0$, hence $M'*N \geq M*N$ trivially.
    Hence we can assume $M$ is loopless.

    Let $\Sigma = (-\Trop(N)) \wedge \Trop(y_e -1)$, then $M*N = \Trop(M) \cdot \Sigma$.
    By assumption, there exists some weight function $q \colon \cB(M') \rightarrow \RR$ such that $\Delta_q(M')$ contains $P(M)$ as a cell.
    We can extend the domain of $q$ to $\binom{E}{k}$ by setting $q(B) = \infty$ for all $B \notin \cB(M')$.
    By \Cref{lem:degrees}, we see that $\Trop(M')\cdot\Sigma = L_q\cdot\Sigma$.
    As $P(M)$ is a cell of $\Delta_q(M)$, there exists some $w \in \RR^E$ such that $M_{q,w} = M$.
    Moreover, as $M$ is assumed to be loopless, we have $w \in L_q$ by \Cref{eq:troplinearspace}.
    By \cite[Corollary~4.4.8]{MaclaganSturmfels2015} we have $\Star_{w}(L_q) = \Trop(M)$, and by \cite[Lemma 5.6]{ClarkeDewarEtAl2025} we have $\Star_{w}(L_q) \cdot \Sigma \leq L_q \cdot\Sigma$.
    Putting this all together gives the required inequality:
    \begin{equation*}
        M * N \ = \ \Trop(M) \cdot \Sigma \  = \ \Star_{w}(L_q) \cdot \Sigma \ \leq \ L_q \cdot \Sigma \ = \ \Trop(M') \cdot \Sigma \ = \ M' * N \, .\qedhere
    \end{equation*}
\end{proof}

\begin{remark}
    We note that the proof of \Cref{prop:flip+relaxation} does not make any use of $\Trop(N)$ beyond that it is a balanced polyhedral complex of the correct dimension.
    As such, we state a more general form of \Cref{prop:flip+relaxation} as follows.
    Let $M$ and $M'$ be as stated, and let $\Sigma$ be a balanced polyhedral complex of dimension $|E| - r(M)$ with only positive multiplicities.
    Then $\Trop(M) \cdot \Sigma \leq \Trop(M') \cdot \Sigma$.
\end{remark}

\begin{remark} \label{rem:brandt+speyer}
    \Cref{prop:flip+relaxation} only gives a partial answer to \Cref{q:flip+weak+order}, as there exists matroids~$M$ and~$M'$ with $M \prec M'$ that do not satisfy the hypothesis of \Cref{prop:flip+relaxation}.
    Brandt and Speyer exhibit two matroids $M, M'$ with $P(M) \subseteq P(M')$ but show there exists no matroid subdivision of~$P(M')$ containing $P(M)$ as a cell~\cite[Theorem 5.2]{BrandtSpeyer2022}.
    However, this does not lead to a counterexample to \Cref{q:flip+weak+order} as far as we know.
\end{remark}

\Cref{ex:representable+flip} and \Cref{ex:relaxation+flip} give two applications of \Cref{prop:flip+relaxation}: the first to perturbations of representable matroids and the second to relaxations of matroids.

\begin{example}\label{ex:representable+flip}
    Let $X$ be a complex $|E| \times r$ matrix with rank $r$, and let $U \subset \mathbb{C}^E$ be its column span.
    For any choice of complex $|E| \times r$ matrix $Y$ and any $t \in \mathbb{C}$,
    set $U_t$ to be the column span of~$X+tY$.
    As the determinant map for square matrices is a polynomial, it follows that there exists a matroid~$M'$ such that $M(U_t) = M'$ for all but finitely many $t$.
    We now show that $M(U) * N \leq M' * N$ for any matroid $N$ of rank $|E| - r + 1$,
    in essence showing that small rotations of linear spaces can only increase flip products.

    To show this, we use the linear spaces to construct a matroid subdivision $\Delta(M')$ of $P(M')$ that contains $P(M)$ as a cell.
    For a $|E| \times r$ matrix $A$,
    we write $[A]_I$ for the minor of a $|E| \times r$ matrix $A$ with rows labelled by $I \in \binom{E}{r}$;
    here we recall that the bases of the row matroid of $A$ are exactly the set of bases of $A$ are exactly the $r$-sets $I$ such that $[A]_I \neq 0$.
    By considering each minor $[X+tY]_I$ to be an element of the field of Puiseux series $\CC\{\!\{t\}\!\}$,
    we now define a weight function $q$ on $\cB' = \cB(M')$ by $q(I) = \nu([X+tY]_I)$, where $\nu$ is the valuation described in \Cref{tropical_background:tropicalisation}.
    This is the tropicalisation of the Pl\"ucker vector of $U_t$ viewed as a linear subspace of $(\CC\{\!\{t\}\!\})^E$, and hence induces a matroid subdivision $\Delta_q(M')$~\cite[Prop 4.4.3]{MaclaganSturmfels2015}.
    Since each $[X+tY]_I$ is a polynomial in $t$, we have that $q(I) = 0$ if and only if $[X]_I \neq 0$, and $q(I) > 0$ otherwise.
    As such, setting $w=0$ gives that $P(M) = P(\cB'_{q,0})$ is a cell of $\Delta_q(M')$.
    Our claim then follows from \Cref{prop:flip+relaxation}.
\end{example}

\begin{remark}
    For any connected multigraph $G=([n],E)$, most\footnote{Each gain assignment can be considered to be a vector in $\mathbb{Z}^E$.
    We now view `most' to mean that the subset $S \subset \mathbb{Z}^E$ where the condition does not hold satisfies the limit equality $\lim_{N \rightarrow \infty} S \cap \{-N, \ldots, N\}^k = 0$.} $\mathbb{Z}^k$-gain assignments $\varphi$ for $G$ produce a matrix $M(G,\varphi)$ with the same row matroid as $[I(G) ~ A]$, where $A$ is a generic $|E| \times k$ complex matrix;
    namely, the matroid union $M_G \vee U_{E,2}$.
    Now, assuming $(G,\phi)$ is a minimally periodic $\mathbb{Z}^k$-gain graph, take $X$ to be the matrix formed from $M(G,\phi)$ by deleting an arbitrary vertex column and $Y = [\mathbf{0}_{|E| \times (n-1)} ~ B]$ with $B \in \mathbb{C}^{|E| \times k}$ chosen to be generic.
    With this, we obtain the following upper bound on $\ccountper{G,\phi}{k}$ courtesy of \Cref{thm:periodic} and \Cref{ex:representable+flip}:
    \begin{equation*}
        \ccountper{G,\phi}{k} = \frac{1}{2} M(G,\phi) \ast M(G,\phi) \leq \frac{1}{2} (M_G \vee U_{E,2}) \ast (M_G \vee U_{E,2}).
    \end{equation*}
\end{remark}

\begin{example}\label{ex:relaxation+flip}
    Let $M$ be a matroid on $E$ and let $X \subset E$ be both a circuit and a hyperplane of~$M$.
    The \emph{circuit-hyperplane relaxation} of $M$ is the matroid $M'$ with bases $\cB(M') = \cB(M) \cup \{X\}$.
    By definition we have $M \preceq M'$.
    We show that the flip product is non-decreasing under circuit-hyperplane relaxations, i.e., $M \ast N \leq M' \ast N$ for any matroid $N$ of rank $|E| - r(M) +1$.

    We demonstrate that we can construct a matroid subdivision $\Delta(M')$ that contains $P(M)$ as a cell; the claim then follows from \Cref{prop:flip+relaxation}.
    Define the weight function $p$ on $\cB' = \cB(M')$ by $p(B) = 0$ for all $I \in \cB(M)$ and $p(X) = 1$.
    Then $\Delta_p(M')$ decomposes into two maximal pieces: $P(M) = P(\cB'_{p,0})$ and $P(\cB'_{p,w})$ where $w_i = 1$ if $i \in X$ and $w_i = 0$ otherwise.
    It is straightforward to verify that $\cB'_{p,w} = \{Y \in \cB(M) \mid |X \cap Y| = r(M)-1\} \cup \{X\}$ form the bases of a matroid, hence $\Delta_p(M')$ is a matroid subdivision.
\end{example}

\begin{remark}
    Circuit-hyperplane relaxations are a special case of the more general family of \emph{matroid relaxations}, introduced in \cite{FerroniSchroeter2024}.
    We refrain from giving a formal definition: both constructions look to give a systematic way of adding bases to a matroid while remaining a matroid, but relaxations allow us to relax more general `stressed sets'.
    If $M'$ is a relaxation of $M$, a similar argument to above says there exists a matroids subdivision $\Delta(M')$ containing $P(M)$ as a cell: see \cite[Theorem~3.30]{FerroniSchroeter2024} for details.
    Applying \Cref{prop:flip+relaxation} gives that $M*N \leq M'*N$  holds for any matroid $N$ of rank $|E| - r(M) +1$.
\end{remark}

\section{Computational results}\label{computational}

In this short section we discuss the possible values of the flip product, and families of matroids where these values are attained. A few easy examples of such values are the following:
\begin{enumerate}
    \item 0 is easy for any size ground set: see \Cref{ex:flip_product_is_zero}.
    \item 1 is also easy to achieve for any size ground set: if $M$ is independent (i.e., $r(M)=|E|$) and $N = U_{E,1}$, then $M \ast N = 1$.
    This is as $\Trop(M) = \mathbb{C}^E$ and so $(-\Trop(M)) \wedge \Trop(N) = \Trop(N) = \mathbb{R} \cdot \mathbf{1}$.
    \item In fact, every positive integer is possible so long as $E$ is allowed to grow. If $M,N$ are uniform then $M \ast N = \binom{|E|-1}{r(M)-1}$, which (by \cite[Lemma 5.8]{ClarkeDewarEtAl2025}, or alternatively by \Cref{prop:flip+relaxation} using the corank subdivision described in \cite{JoswigSchroter2017}) is the largest possible value over all pairs of matroids with the same ground set and ranks $r(M),r(N)$.
    Hence, we have $U_{n,n-1} \ast U_{n,2} = n-1$ for each positive integer $n$.
    \item By the previous point, the maximal value for $M \ast N$ for a fixed ground set $E$ is when both $M$ and $N$ are uniform and one has rank $\lceil|E|/2\rceil$: in such a case, we achieve
    \begin{equation*}
        M \ast N = \binom{|E|-1}{\lceil|E|/2\rceil - 1}   \approx 2^{|E|-1}.
    \end{equation*}
\end{enumerate}

It is possible to enumerate all matroids on small ground sets. For instance, this functionality can be found in the \texttt{Matroids} package \cite{Chen2019} for the computer algebra system \textit{Macaulay2} \cite{GraysonStillman}. Let $\mathcal M_{n, r}$ denote a set of isomorphism class representatives of rank $r$ matroids on $[n]$. For any collection of positive integers $k_1, k_2, p$, let $n = k_1 + k_2 - 1$ and define
\begin{equation*}
    h_{k_1, k_2}(p) = \Big|\big\{(M, N, \sigma) : M \in \mathcal M_{n, k_1},\, N \in \mathcal M_{n, k_2},\, \sigma \in S_n,\, M * \sigma(N) = p \big\} \Big|,
\end{equation*}
where $\sigma \in S_n$ is a permutation that acts by permuting the ground set of $N$. Note that the value $h_{k_1, k_2}(p)$ does not depend on the choice of representatives $\mathcal M_{n,r}$.
These values have the property that $h_{k_1, k_2}(p) > 0$ if and only if there exist matroids $M, N$ of rank $k_1, k_2$ respectively on $n = k_1 + k_2 - 1$ such that $M*N = p$.

\begin{table}[t]
    \centering
    \resizebox{0.9\textwidth}{!}{
        \begin{tabular}{lr|rrrrrrrrrrrr}\toprule
            & & $p$ & & & & & & & & & & \\
            $n$ & $(k_1, k_2)$ &0 &1 &2 &3 &4 &5 &6 &7 &8 &9 &10 \\\midrule
            1 &(1,1) &0 &1 & & & & & & & & & \\
            2 &(1,2) &2 &2 & & & & & & & & & \\
            3 &(1,3) &12 &6 & & & & & & & & & \\
              &(2,2) &32 &16 &6 & & & & & & & & \\\midrule
            4 &(1,4) &72 &24 & & & & & & & & & \\
              &(2,3) &414 &174 &60 &24 & & & & & & & \\\midrule
            5 &(1,5) &480 &120 & & & & & & & & & \\
              &(2,4) &5208 &1416 &768 &288 &120 & & & & & & \\
              &(3,3) &11724 &3864 &2596 &1192 &508 &276 &120 & & & & \\\midrule
            6 &(1,6) &3600 &720 & & & & & & & & & \\
              &(2,5) &66624 &16296 &9792 &4248 &1680 &720 & & & & & \\
              &(3,4) &335160 &85116 &78172 &51624 &32808 &18128 &14372 &7772 &3824 &1584 &720 \\\midrule
            7 &(1,7) &30240 &5040 & & & & & & & & & \\
              &(2,6) &916704 &163152 &119376 &61488 &28080 &11520 &5040 & & & & \\
            \bottomrule
        \end{tabular}
    }
    \caption{Values of $h_{k_1, k_2}(p)$: the number of pairs of matroids on the ground set $[k_1 + k_2 - 1]$, up to permutations, with ranks $k_1$, $k_2$ and whose flip product equals~$p$.}
    \label{tab: small hadamard products with permutation}
\end{table}

In \Cref{tab: small hadamard products with permutation}, we write down values of $h_{k_1, k_2}$ for small values of $k_1, k_2$. These computations lead us to the following conjectures.

\begin{conjecture}\label{conj:h}
    For any pair of positive integers $k_1, k_2$ let $n = k_1 + k_2 - 1$. We have the following:
    \begin{enumerate}
        \item \label{conj:h1} $h_{k_1, k_2}(p) \ge 1$ for all $0 \le p \le U_{n,k_1} * U_{n,k_2}$ and $n \ge 2$,
        \item \label{conj:h2} for $n \ge 3$, we have $h_{k_1, k_2}(p) \le n!$ if and only if $p = U_{n,k_1} * U_{n,k_2}$.
    \end{enumerate}
\end{conjecture}

We observe that the reverse direction for \Cref{conj:h}\ref{conj:h2} follows from a straightforward nbc-basis count combined with \Cref{thm:nbc}.
We note that it may be possible to prove \Cref{conj:h}\ref{conj:h1} using nbc-bases and \Cref{thm:nbc}. A straightforward computation of the flip-product with an appropriate uniform matroid verifies this conjecture for all matroids in $\mathcal M_{n,k}$ for with $n \le 8$ and $(n,k) \in \{(9,1), (9,2), (9,3)\}$.

In \cite{JacksonOwen2019},
Jackson and Owen conjectured that for a minimally rigid $n$-vertex graph, the value~$\ccount{G}$ (described in \Cref{sec:intro}) is bounded below by $2^{n-3}$.
Inspired by \Cref{thm:originalrigid},
Matt Larson asked whether this bound holds for all matroids:

\begin{question}\label{q:M*M}
    Let $M$ be a matroid on $E$ such that $2r(M) = |E|+1$.
    If $M \ast M > 0$, then is $M \ast M \geq 2^{r(M)-1}$?
\end{question}

When $M$ is a graphic matroid, this is precisely the conjecture of Jackson and Owen, which is open.
To investigate this, we record in \Cref{tab: small hadamard self products} the number of rank $k$ matroids $M$ on the ground set $[n]$ with $M \ast M =p$ for small values of $k$ and $p$, where $n = 2k-1$.
For arbitrary matroids, we observe from \Cref{tab: small hadamard self products} that there is a counterexample to \Cref{q:M*M} which we describe below.

\begin{table}[t]\centering
    \begin{tabular}{l|rrrrrrrrrrrrr}\toprule
        &$p$ & & & & & & & & & & & \\
        $(n,k)$ &0 &1 &2 &4 &6 &8 &10 &12 &14 &16 &18 &20 \\\midrule
        $(1,1)$ &0 &1 & & & & & & & & & & \\
        $(3,2)$ &2 &0 &1 & & & & & & & & & \\
        $(5,3)$ &10 &0 &0 &2 &1 & & & & & & & \\
        $(7,4)$ &71 &0 &0 &0 &1 &15 &1 &12 &3 &2 &2 &1 \\
        \bottomrule
    \end{tabular}
    \caption{Number of matroids $M$ on $[n]$ of rank $r$ with $M*M = p$}\label{tab: small hadamard self products}
\end{table}

\begin{example}\label{ex:jacksonowencounterex}
    Consider the column matroid $M$ of the following matrix over $\FF_2$:
    \begin{equation*}
    \begin{bmatrix}
        1 & 1 & 1 & 0 & 0 & 0 & 0 \\
        1 & 0 & 0 & 1 & 1 & 0 & 0 \\
        0 & 1 & 0 & 1 & 0 & 1 & 0 \\
        1 & 1 & 0 & 1 & 0 & 0 & 1
    \end{bmatrix} \, .
    \end{equation*}
    It is rank 4 on seven elements, so satisfies $2r(M) = |E|+1$.
    However, a computer check demonstrates that $M*M = 6 < 2^{r(M)-1}$.
    It can be seen in \Cref{tab: small hadamard self products} that this is the smallest counterexample to \Cref{q:M*M}.
\end{example}

The matroid described in \Cref{ex:jacksonowencounterex} is an example of a \emph{sparse paving matroid}, namely, a matroid~$M$ where every subset of $r(M)$ elements is either a basis or a circuit-hyperplane.
It is conjectured that most matroids are sparse paving matroids \cite{MayhweNewmanEtAl2011}.
Because of this, we suspect that \Cref{q:M*M} can fail infinitely often, potentially with high probability as the ground set grows.

\begin{remark}
    Supplementary software to this paper --- including a \textit{Macaulay2} implementation of \Cref{thm:main} --- can be found at the Zenodo repository \cite{supportingCode}.
\end{remark}

\addcontentsline{toc}{section}{Acknowledgements}
\section*{Acknowledgements}
We thank Tony Nixon, Yue Ren and Benjamin Schr\"oter for their helpful conversations on the topic.

This research was funded in part by the Austrian Science Fund (FWF)
10.55776/I6233. For open access purposes, the authors have applied a CC BY public copyright license to any author accepted manuscript version arising from this submission.

S.\,D.\ was supported by the Heilbronn Institute for Mathematical Research, and was partially supported by the FWO grants G0F5921N (Odysseus) and G023721N, and by the KU Leuven grant iBOF/23/064.
G.\,G.\ was supported by the Austrian Science Fund (FWF) 10.55776/I6233.
B.\,S.\ was supported by the EPSRC grant EP/X036723/1.
D.\,G.\,T.\ was supported by EPSRC grant EP/W524414/1.

\bibliographystyle{plainurl}
\bibliography{ref}

\appendix

\section{Tropical geometry background}
\label{tropical_background}

In this appendix section we cover all the necessary tropical geometry background material required throughout the paper.

\subsection{Tropicalisation and tropical varieties} \label{tropical_background:tropicalisation}

Let $\mathbb{K}$ be an algebraically closed field with \textit{valuation} $\nu$;
i.e., a map $\nu \colon \mathbb{K} \setminus \{0\} \rightarrow \RR$ such that for all $a,b\in \mathbb{K}$ we have:
\begin{itemize}
    \item $\nu(ab) = \nu(a) + \nu(b)$,
    \item $\nu(a+b) \ge \min\{\nu(a), \nu(b)\}$ with equality if $\nu(a) \neq \nu(b)$.
\end{itemize}
A valuation is said to be \textit{non-trivial} if it is not identically zero.
If our valuation is non-trivial,
we define the \emph{tropicalisation} of an algebraic set $X \subseteq \mathbb{K}^n$ to be the Euclidean closure of the coordinate-wise valuation of its points with non-zero coordinates:
\begin{equation*}
    \Trop(X) := \overline{\Big\{ \big(\nu(x_1), \nu(x_2), \dots, \nu(x_n) \big) : x \in X, ~ x_i \neq 0 ~ \forall i \in [n] \Big\}} \subseteq \RR^n.
\end{equation*}
We now say that $\Trop(X)$ is a \emph{tropical variety}.
If $I \subseteq \mathbb{K}[x_1^{\pm}, \dots, x_n^{\pm}]$ is the set of Laurent polynomials that vanish on $X \cap (\mathbb{K}\setminus\{0\})^n$,
we set $\Trop(I) := \Trop(X)$.
For a single Laurent polynomial $f \in \mathbb{K}[x_1^{\pm}, \dots, x_n^{\pm}]$ we use the shorthand $\Trop (f) = \Trop (\langle f \rangle)$.

Typically we work with polynomials over the field $\CC$, which does not naturally come equipped with a non-trivial valuation. Thus we extend $\CC$ to field of Puiseux series $\CC\{\!\{t\}\!\}$ with the valuation
\begin{equation*}
    \nu \left( \sum_{i=k}^\infty a_i t^{i/n} \right) := k/n  \quad \text{where} \quad a_k \neq 0.
\end{equation*}
Specifically,
given an algebraic set $X \subset \CC^n$ where $X \cap (\mathbb{C}\setminus \{0\})^n$ is the vanishing set of the ideal $I \subseteq \CC[x_1^{\pm}, \dots, x_n^{\pm}]$,
we set $\widetilde I := I \cdot \CC\{\!\{t\}\!\}[x_1^{\pm}, \dots, x_n^{\pm}]$ and then define $\Trop(X) := \Trop (\widetilde{I})$.

\begin{example}
    The linear form $x_1 - 1 \in \CC[x_1, \dots, x_n]$ has zero set $\{(1,x_2,\ldots,x_n) : x_i \in \mathbb{C}\}$.
    Thus its tropicalisation is $\Trop(x_1 - 1) = \{0\} \times \RR^{n-1}$.
\end{example}

\subsection{Balanced polyhedral complexes}\label{tropical_background:bcp}

A \emph{polyhedral complex} $\Sigma$ in $\mathbb{R}^n$ is a collection of (convex) polyhedra in $\mathbb{R}^n$ such that:
\begin{itemize}
    \item If $\sigma \in \Sigma$ then every face of $\sigma$ is contained in $\Sigma$,
    \item If $\sigma,\tau \in \Sigma$ then $\sigma \cap \tau$ is a face of $\sigma$ and $\tau$.
\end{itemize}
The \emph{support} of a polyhedral complex~$\Sigma$ is the union $|\Sigma| := \bigcup_{\sigma \in \Sigma} \sigma \subseteq \mathbb{R}^n$.
The \emph{dimension} of a polyhedral complex is the maximum of the dimensions of its polyhedra.
The inclusion-wise maximal polyhedra of $\Sigma$ are the \emph{maximal polyhedra} or \emph{maximal cells} of~$\Sigma$.
A polyhedral complex is called \emph{pure} if each of its maximal polyhedra have the same dimension.
The \emph{lineality space} of a polyhedral complex $\Sigma$ is the largest affine subspace $L \subset \mathbb{R}^n$ such that for all points $x \in |\Sigma|$ and all $y \in L$ we have $x+y \in |\Sigma|$.
For each face $\sigma \in \Sigma$, the \emph{relative interior} of $\sigma$ is
\begin{equation*}
{\rm relint}(\sigma) = \{x \in \sigma : x \notin \tau \text{ for any proper face } \tau \subsetneq \sigma\}.
\end{equation*}
A polyhedral complex is called a \emph{fan} if all its faces are convex cones,
and a polyhedral complex is \emph{rational} if each polyhedra $P \in \Sigma$ is defined by rational hyperplanes.
A \emph{weighted polyhedral complex $\Sigma$} is a polyhedral complex together with the data of multiplicities $\mult_\Sigma(\sigma) \in \ZZ$ for each maximal polyhedron $\sigma \in \Sigma$.

With this terminology, we can now state the following definition.

\begin{definition}\label{def:balanced}
    Let $\Sigma$ be a $d$-dimensional pure rational weighted polyhedral complex.
    For each polyhedron $\eta \in \Sigma$, define the lattice $N_\eta := \ZZ^n \cap \Span(\eta - u)$ for some arbitrary element $u \in \eta$.
    Fix a $(d-1)$-dimensional cell $\tau$ and, for each $d$-dimensional cell $\sigma$ containing $\tau$,
    let $v_\sigma\in\ZZ^n$ be a vector such that $\ZZ\cdot v_\sigma + N_\tau = N_\sigma$.
    We say that $\Sigma$ is \emph{balanced at} $\tau$ if $\sum_{\sigma\supsetneq\tau}\mult_\Sigma(\sigma)v_\sigma\in N_\tau$, and that $\Sigma$ is a \emph{balanced polyhedral complex} (or \emph{tropical cycle}) if it is balanced at every $(d-1)$-dimensional cell.
\end{definition}

It is good to keep in mind the following key examples:
\begin{enumerate}
    \item Every tropical variety is a balanced polyhedral complex \cite[Theorem 3.3.5]{MaclaganSturmfels2015}.
    \item By assigning multiplicity 1 to each maximal cell, every Bergman fan is a \emph{balanced fan}, a balanced polyhedral complex that is also a fan \cite[Theorem 4.4.5]{MaclaganSturmfels2015}.
    \item We can extend our definition of balanced polyhedral complexes to allow for polyhedral complexes contained within spaces such as $\mathbb{R}^n / \mathbb{R} \cdot \mathbf{1}$. Because of this, we consider the projective Bergman fan to also be a balanced fan.
\end{enumerate}

\subsection{Stable intersection}\label{def:stableIntersection}

Let $\Sigma_1,\Sigma_2$ be two balanced polyhedral complexes in $\RR^n$.
Their \emph{stable intersection} is the pure rational polyhedral complex consisting of the polyhedra
\begin{equation*}
    \Sigma_1\wedge\Sigma_2 \coloneqq \Big\{\sigma_1\cap\sigma_2\bigmid\sigma_1\in\Sigma_1, \sigma_2\in\Sigma_2 \, , \, \dim(\sigma_1+\sigma_2)=n \Big\} \, .
\end{equation*}
We make the stable intersection $\Sigma_1 \wedge \Sigma_2$ into a balanced polyhedral complex by assigning the following multiplicity to each maximal cell $\sigma_1\cap\sigma_2\in\Sigma_1\wedge\Sigma_2$:
\begin{equation*}
    \mult_{\Sigma_1\wedge\Sigma_2}(\sigma_1\cap\sigma_2) \coloneqq \sum_{\tau_1,\tau_2}\mult_{\Sigma_1}(\tau_1)\mult_{\Sigma_2}(\tau_2)[ \mathbb{Z}^n : N_{\tau_1}+N_{\tau_2}],
\end{equation*}
where the sum is over all maximal cells $\tau_i\in\Sigma_i$ with $\sigma_1 \cap \sigma_2 \subseteq \tau_i$ for $i=1,2$, and $\tau_1\cap (\tau_2+\varepsilon\cdot v)\neq\emptyset$ for a fixed generic $v\in\RR^n$ and $\varepsilon>0$ sufficiently small.
The integer $[\mathbb{Z}^n : N_{\tau_1}+N_{\tau_2}]$ denotes the \emph{index} of the sublattice $N_{\tau_1} + N_{\tau_2} \subseteq \mathbb{Z}^n$.
The multiplicity of the stable intersection does not depend of the choice of $v$, hence it is well-defined.

The assumption that $\Sigma_1$ and $\Sigma_2$ are balanced is important, since it implies the following two properties.
Firstly, if the stable intersection $\Sigma_1 \wedge \Sigma_2$ is 0-dimensional, then the summation
\begin{equation*}
    \sum_{\tau} \mult_{\Sigma_1 \wedge (\Sigma_2+u)} (\tau)
\end{equation*}
taken over each maximal cell $\tau \in  \Sigma_1 \wedge (\Sigma_2+u)$, is independent of the choice of $u$.
Secondly, it implies that stable intersection is associative, and so the stable intersection $\Sigma_1 \wedge \cdots \wedge \Sigma_k$ is well-defined for any choice of balanced polyhedral complexes $\Sigma_1,\ldots,\Sigma_k$.

\subsection{Recession fans and stars}

We now describe two balanced fans that can be formed from a balanced polyhedral complex $\Sigma$: recession fans and stars.

\begin{definition}\label{def:recessionfan}
Given a polyhedra $\sigma = \{ x \in \RR^n \mid  Ax \le b\}$ for some matrix $A \in \RR^{d \times n}$ and $b \in \RR^d$, we define the recession cone of $\sigma$ as $\rec(\sigma) := \{x \in \RR^n \mid Ax \le \mathbf{0} \}$.
From this, we define $\rec(\Sigma)$ to be the union of all cones $\rec(\sigma)$ taken over polyhedra $\sigma \in \Sigma$.
The set $\rec(\Sigma)$ can always be given the structure of a polyhedral fan, which we then call the \emph{recession fan} of $\Sigma$.
Additionally, $\rec(\Sigma)$ is a balanced fan when assigned the multiplicities $\mult_{\rec(\Sigma)}(\tau)=\sum_{\rec(\tau')=\tau}\mult_\Sigma(\tau')$.
\end{definition}

If $\Sigma'$ is another polyhedral complex with the same support as $\Sigma$, then the recession fans $\rec(\Sigma)$ and $\rec(\Sigma')$ have the same support, hence the recession fan is well-defined on supports.

\begin{definition}\label{def:star}
    Let $w$ be a point in the support $|\Sigma|$.
    The \emph{star} of $\Sigma$ at $w$ is the balanced fan $\Star_w(\Sigma)$ with maximal cells
    \begin{equation*}\sigma_w := \left\{\lambda(v - w) \mid v \in \sigma \, , \, \lambda \geq 0\right\}, \quad \text{ for each maximal cell $\sigma \in \Sigma$ containing } w \in \sigma ,
    \end{equation*}
    where each cone $\sigma_w$ has multiplicity $\mult_\Sigma(\sigma)$.
\end{definition}

Intuitively, $\mathrm{star}_w(\Sigma)$ is $\Sigma$ viewed locally around $w$. In the case that $\Sigma$ is the Bergman fan of a matroid $M$, its stars are given by Bergman fans of the \textit{initial matroids $M_{p,w}$}. See \Cref{def: matroid subdivision}.

\begin{proposition}[{\cite[Corollary~4.4.8]{MaclaganSturmfels2015}}]
    Let $M$ be a matroid of rank $r$ on $[n]$ and $p \in \RR^n$ be tropical Pl\"ucker vector. Let $\sigma \in L_p$ be a cone of the tropical linear space $L_p = \{w \in \RR^n : M_{p, w} \text{ has no loops}\}$ and $w \in \sigma$ be a point in its relative interior. Then $\Star_w(L_p) = \Trop(M_{p,w})$, in particular, we have
    \begin{equation*}
        \Star_w(\Trop(M)) = \Trop(M_{p,w})
        \quad \text{for} \quad
        p_B = 
        \begin{cases}
            0 & \text{if } B \in \mathcal B(M),\\
            \infty & \text{if } B \notin \mathcal B(M).
        \end{cases}
    \end{equation*}
\end{proposition}

\section{Proof of \texorpdfstring{\Cref{lem:removebridge} and \Cref{lemma:tropical_untangling}}{bridge removal and tropical untangling lemmas}}\label{subsec:defferedlemmas}

Before we prove our previously deferred key lemmas,
we require the following technical result regarding projections of Bergman fans.

\begin{lemma}[{\cite[Lemma 3.8]{FrancoisRau2013}}]\label{lem:FrancoisRau2013}
    Let $M$ be a loopless matroid with ground set $E$ and let $F \subset E$ be a flat of $M$ where $r(M\setminus F) = r(M)$ (e.g., $F=\{e\}$ where $e$ is not a coloop).
    Let
    \begin{equation*}
        \pi_F \colon \Trop (M) \rightarrow \Trop (M\setminus F), ~ (x_e)_{e \in E} \mapsto (x_e)_{e \in E \setminus F}.
    \end{equation*}
    Then $\pi_F$ is surjective.
    Moreover, $\pi_F$ is injective when restricted to the union of the relative interiors of maximal cones of $\Trop(M)$ not containing $\chi_S$ in their linear span for every non-empty subset $S\subseteq F$ (e.g. not containing $\{0\}^{E\setminus e}\times\RR$ in the case $F=\{e\}$).
\end{lemma}

\begin{proof}
    The only difference with \cite[Lemma 3.8]{FrancoisRau2013} is the explicit condition on when the $\pi_F$ is injective.
    However, their proof notice that this happens in the relative interior of cones $\sigma_\mathcal{G}$ corresponding to maximal chains of flats $\mathcal{G}: \emptyset=G_0\subsetneq G_1\subsetneq \cdots \subsetneq G_r=E$ such that $\emptyset=G_0\setminus F\subseteq G_1\setminus F\subseteq \cdots \subseteq G_r\setminus F=E\setminus F$ is a maximal chain of flats of $M\setminus F$;
    that is, when $G_{i-1}\setminus F\neq G_i\setminus F$ for all $i\in\{1,\dots,r\}$.
    Now, this is equivalent to saying that $G_i\setminus G_{i-1}\nsubseteq F$ for all $i\in\{1,\dots,r\}$. If we had $S:=G_j\setminus G_{j-1}\subseteq F$ for some $j$, then
    \begin{equation*}
        \chi_S\in\lin(\sigma_\mathcal{G})=\lin\bigl(\bigl\{\chi_{G_i\setminus G_{i-1}}:i\in\{1,\dots,r\}\bigr\}\bigr).
    \end{equation*}
    Conversely, if $\chi_S\in\lin(\sigma_\mathcal{G})$ for some non-empty $S\subseteq F$ then, being $\{\chi_{G_i\setminus G_{i-1}}:i\in\{1,\dots,r\}\}$ a basis of $\lin(\sigma_\mathcal{G})$ with disjoint supports, we have that there exists $j\in\{1,\dots,r\}$ such that $G_j\setminus G_{j-1}\subseteq S\subseteq F$.
    Hence, $G_i\setminus G_{i-1}\nsubseteq F$ for all $i\in\{1,\dots,r\}$ is equivalent to $\lin(\sigma_\mathcal{G})$ not containing $\chi_S$ for every non-empty subset $S\subseteq F$.
\end{proof}

We recall the statement of \Cref{lem:removebridge}:

\removebridge*

\begin{proof}
    If $M$ has a loop $\ell$ then, since $e$ is not a loop of $M$, we have that $\ell$ is a loop $M\setminus e$.
    Similarly, if $N$ has a loop $\ell$ then, since $e$ is coloop of $N$ and thus cannot be a loop, we have that $\ell$ is a loop $N\setminus e$.
    In either case, it follows from \Cref{prop:loopbad} that $M \ast N = 0$ and $(M \setminus e) \ast (N \setminus e) = 0$.
    We may now assume that $M$ and $N$ are loopless.

    As $e$ is a coloop of $N$, then $N=(N\setminus e) \oplus U_{\{e\},1}$ and
    \begin{equation}\label{eq:removebridge1}
        \Trop(N)=\Trop(N\setminus e)\times\Trop(U_{\{e\},1}) = \Trop (N\setminus e) \times \mathbb{R}.
    \end{equation}
    Now choose generic $w=(v,v_e)\in \RR^{E\setminus e}\times\RR$ and fix an element $\epsilon \in E \setminus \{e\}$.
    By \Cref{eq:removebridge1}, we have
    \begin{equation*}
        -\Trop(N) + w = (-\Trop (N\setminus e) +v ) \times \mathbb{R}.
    \end{equation*}
    Combining this with the representation of the flip product as a cardinality (see \Cref{eq:FPasCardinality}),
    we have that
    \begin{align}\label{eq:M ast N bridge}
        M\ast N &= \Big| \Trop(M)\cap(-\Trop(N)+w)\cap\Trop(x_\epsilon-1)\Big| \nonumber \\
        &=\Bigg| \Trop(M) \cap \bigg( \Big[ (-\Trop (N\setminus e)+v )\cap \Trop(y_\epsilon -1) \Big] \times \mathbb{R}\bigg) \Bigg| ,
    \end{align}
    (where $\Trop(x_\epsilon-1) \subset \mathbb{R}^E$ and $\Trop(y_\epsilon-1) \subset \mathbb{R}^{E \setminus \{e\}}$)
    and
    \begin{equation}\label{eq:setB}
        (M\setminus e)\ast (N\setminus e) = \Big| \Trop(M\setminus e)\cap(-\Trop(N\setminus e)+v)\cap\Trop(y_\epsilon -1) \Big|,
    \end{equation}
    and all points in the respective intersection lie in the relative interior of maximal cones of each complex.
    With this in mind, we define the finite sets
    \begin{align*}
        A &:= \Trop(M) \cap \bigg( \Big[ (-\Trop (N\setminus e)+v )\cap \Trop(y_\epsilon -1) \Big] \times \mathbb{R}\bigg) \\
        B &:= \Trop(M\setminus e)\cap(-\Trop(N\setminus e)+v)\cap\Trop(y_\epsilon -1).
    \end{align*}
    From \Cref{eq:M ast N bridge,eq:setB},
    $M \ast N = |A|$ and $(M\setminus e)\ast (N\setminus e) = |B|$.
    Let $\pi_e := \pi_{\{e\}}$ be the map defined in \Cref{lem:FrancoisRau2013} for $M$ and $F=\{e\}$.
    Since $\pi_e(A) = B$ and each point of $A$ is contained within the relative interior of the cones of $\Trop(M)$, it now suffices to show that no point in~$A$ is contained in a maximal cone of $\Trop(M)$ that contains $\{0\}^{E\setminus e}\times\RR$ in its linear span.

    Suppose otherwise, and set $(p,p_e) \in A$ to be a point that is contained in a maximal cone $\sigma$ of $\Trop(M)$ which in turn contains $\{0\}^{E\setminus e}\times\RR$ in its linear span.
    Since $(p,p_e)$ is contained within the relative interior of $\sigma$,
    there exists $r_0>0$ such that $(p,p_e+r) \in \sigma$ for each $0<r<r_0$.
    However, this now implies that $(p,p_e+r) \in A$ for each $0<r<r_0$,
    contradicting that $A$ is a finite set.
    This concludes the proof of the lemma.
\end{proof}

We recall the statement of \Cref{lemma:tropical_untangling}:

\tropicaluntangling*

\begin{proof}
    Fix the following Bergman fans:
    \begin{align*}
        \Sigma_1 &:= \Trop(M_1\setminus \epsilon), & \widetilde{\Sigma}_1 &:= \Trop(M_1), & \Sigma_2 &:= \Trop(M_2),  \\
        T_1 &:= \Trop(N_1), & T_2 &:= \Trop(N_2\setminus \epsilon), & \widetilde{T}_2 &:= \Trop(N_2).
    \end{align*}
    By construction,
    \begin{equation*}
        \widetilde{\Sigma}_1\times\Sigma_2=\Trop(M_1\oplus M_2), \qquad T_1\times\widetilde{T}_2=\Trop(N_1\oplus N_2).
    \end{equation*}
    Choose a generic point $(\lambda_1,\lambda_\epsilon, \lambda_2)\in\RR^{E_1\setminus \epsilon}\times\RR\times\RR^{E_2\setminus \epsilon}$.
    With this, each of the pairs $(\Sigma_1,\lambda_1-T_1)$, $(\Sigma_2,\lambda_2-T_2)$ and $(\widetilde{\Sigma}_1\times\Sigma_2,(\lambda_1,\lambda_\epsilon,\lambda_2)-(T_1\times\widetilde{T}_2))$ intersect transversally in the relative interior of maximal cones. Now consider the sets
    \begin{align*}
        S &= \Big\{(x_1,x_2,y_1,y_2)\in \Sigma_1\times\Sigma_2\times T_1\times T_2: x_i+y_i=\lambda_i , \ i \in \{1,2\} \Big\},\\
        \widetilde{S} &= \Big\{(x_1,x_\epsilon,x_2,y_1,y_\epsilon,y_2)\in \widetilde{\Sigma}_1\times\Sigma_2\times T_1\times \widetilde{T}_2: x_i+y_i=\lambda_i , \ i \in \{1,2,\epsilon\} \Big\}.
    \end{align*}
    Taking just the $x$-coordinates, $S$ and $\widetilde{S}$ correspond to the intersections
    \[
        (\Sigma_1\times\Sigma_2)\cap\bigl((\lambda_1, \lambda_2) - T_1\times T_2\bigr)
        \quad \text{and} \quad
        (\widetilde{\Sigma}_1\times\Sigma_2)\cap\bigl((\lambda_1,\lambda_\epsilon, \lambda_2) - T_1\times \widetilde{T}_2\bigr),
    \]
    respectively. 
    Consider as well the equivalence relations $\sim_1$ and $\sim_2$ on $\widetilde{S}$ and $S$, respectively, given by 
    \begin{align*}
        (x_1, x_\epsilon,x_2,y_1, y_\epsilon,y_2)&\sim_1 (x_1+a\bo,x_\epsilon+a,x_2+a\bo,y_1-a\bo, y_\epsilon-a ,y_2-a\bo), \\
        (x_1,x_2,y_1,y_2)&\sim_2 (x_1+b\bo,x_2+c\bo,y_1-b\bo,y_2-c\bo)
    \end{align*}
    for any $a,b,c\in\RR$, where $\bo=(1,\ldots,1)$ in the corresponding Euclidean space. Note that:
    \begin{equation*}
        (M_1 \oplus M_2) \ast (N_1  \oplus  N_2)=\Big|\widetilde{S}/\!\!\sim_1\Big|, \qquad
        \Big( (M_1 \setminus \epsilon) \ast N_1 \Big) \Big( M_2 \ast (N_2 \setminus \epsilon) \Big)=\Big|S/\!\!\sim_2\Big|.
    \end{equation*}
    Hence, it is enough to give injective functions $\varphi \colon \widetilde{S}/\!\!\sim_1\longrightarrow S/\!\!\sim_2$ and $\psi \colon S/\!\!\sim_2\longrightarrow \widetilde{S}/\!\!\sim_1$.

    By \cite[Proposition 2.25]{Shaw2013}, there exist functions $g \colon T_2\longrightarrow\RR$ and $h \colon \Sigma_1\longrightarrow\RR$ such that the following holds:
    \begin{enumerate}[label=(\alph*)]
        \item $g$ (respectively, $h$) is linear when restricted to each cone of $T_2$ (respectively, $\Sigma_1$);
        \item the maps
        \begin{equation*}
            T_2\rightarrow \widetilde{T}_2, ~ y \mapsto \bigl(y, g(y)\bigr), \qquad \Sigma_1\rightarrow\widetilde{\Sigma}_1, ~ x \mapsto \bigl(x, h(x)\bigr)
        \end{equation*}
        are right inverses of the projections $\widetilde{T}_2\longrightarrow T_2$ and $\widetilde{\Sigma}_1\longrightarrow\Sigma_1$, respectively;
        \item $g(y+a\bo)=g(y)+a$ and $h(x+a\bo)=h(x)+a$ for any $a\in \RR,\ y\in T_2,\ x\in \Sigma_1$.
    \end{enumerate}
    With these two functions, we make the following key observation we exploit later on.
    Choose any $(x_1, x_\epsilon,x_2,y_1, y_\epsilon,y_2)\in \widetilde{S}$.
    Then $(x_1,x_2,y_1,y_2)\in S$, and so $x_1$ is in the relative interior of a maximal cell of $\Sigma_1$ and $y_2$ is in the relative interior of a maximal cell of $T_2$. An application of \Cref{lem:FrancoisRau2013} then gives that $x_\epsilon=h(x_1)$ and $y_\epsilon=g(y_2)$ must hold.

    Define the (possibly poorly-defined) maps
    $\varphi \colon \widetilde{S}/\!\!\sim_1\longrightarrow S/\!\!\sim_2$ and $\psi \colon S/\!\!\sim_2 \longrightarrow \widetilde{S}/\!\!\sim_1$ as follows:
    \begin{align*}
        \varphi(x_1,x_\epsilon,x_2,y_1, y_\epsilon,y_2)&=(x_1+x_\epsilon\bo,x_2-x_\epsilon\bo,y_1-x_\epsilon\bo, y_2+x_\epsilon\bo),\\
        \psi(x_1,x_2,y_1,y_2)&=(x_1-h(x_1)\bo, 0, x_2+(g(y_2)-\lambda_\epsilon)\bo, y_1+h(x_1)\bo,\lambda_\epsilon,y_2-(g(y_2)-\lambda_\epsilon)\bo).
    \end{align*}
    We now prove both $\varphi$ and $\psi$ are well-defined and injective.

    \begin{claim}
        $\varphi$ is well-defined and injective.
    \end{claim}

    \begin{claimproof}
        Choose any $(x_1, x_\epsilon,x_2,y_1, y_\epsilon,y_2)\in \widetilde{S}$.
        It is clear that $(x_1 + x_\epsilon\bo, x_2-x_\epsilon\bo, y_1-x_\epsilon\bo, y_2+x_\epsilon\bo)\in S$, and
        \begin{align*}
            &\Big((x_1+a\bo) + (x_\epsilon+a)\bo, (x_2+a\bo)-(x_\epsilon+a)\bo, (y_1-a\bo) - (x_\epsilon+a)\bo, (y_2-a\bo) + (x_\epsilon+a)\bo\Big)\\
            = \ &\Big((x_1+x_\epsilon\bo) + 2a\bo, x_2 - x_\epsilon\bo,(y_1-x_\epsilon\bo) - 2a\bo, y_2+x_\epsilon\bo\Big)\\
            \sim_2 \ &\Big(x_1 + x_\epsilon\bo, x_2 - x_\epsilon\bo,y_1-x_\epsilon\bo, y_2 + x_\epsilon\bo\Big),
        \end{align*}
        proving that $\varphi \colon \widetilde{S}/\sim_1\longrightarrow S/\sim_2$ is indeed well-defined.

        Now assume there exists $(x_1',x_\epsilon',x_2',y_1', y_\epsilon',y_2')\in \widetilde{S}$ such that
        \begin{equation*}
            \Big(x_1+x_\epsilon\bo,x_2-x_\epsilon\bo,y_1-x_\epsilon\bo, y_2+x_\epsilon\bo \Big)=\Big(x_1'+x_\epsilon'\bo+a\bo,x_2'-x_\epsilon'\bo+b\bo,y_1'-x_\epsilon'\bo-a\bo, y_2'+x_\epsilon'\bo-b\bo \Big)
        \end{equation*}
        for some $a,b \in \mathbb{R}$.
        Then $x_1'=x_1+(x_\epsilon-x_\epsilon'-a)\bo$, which implies that
        \begin{equation*}
            x_\epsilon-x_\epsilon'=h(x_1)-h(x_1')=h(x_1)-h(x_1+(x_\epsilon-x_\epsilon'-a)\bo) =-(x_\epsilon-x_\epsilon')+a.
        \end{equation*}
        Thus $a=2(x_\epsilon-x_\epsilon')$. Similarly, $y_2'=y_2+(x_\epsilon-x_\epsilon'+b)\bo$ implies that $y_\epsilon'-y_\epsilon=x_\epsilon-x_\epsilon'+b$, but
        \begin{equation*}
            y_\epsilon'-y_\epsilon=\lambda_\epsilon-x_\epsilon'-\lambda_\epsilon+x_\epsilon=x_\epsilon-x_\epsilon'
        \end{equation*}
        and thus $b=0$.
        Substituting the values of $a$ and $b$ we get
        \begin{align*}
            x_1'&=x_1-(a/2)\bo & y_1'&=y_1+(a/2)\bo\\
            x_2'&=x_2-(a/2)\bo & y_2'&=y_2+(a/2)\bo\\
            x_\epsilon'&=x_\epsilon+x_\epsilon'-x_\epsilon=x_\epsilon-a/2&
            y_\epsilon'&=y_\epsilon+y_\epsilon'-y_\epsilon=y_\epsilon+x_\epsilon-x_\epsilon'\\
            &&&=y_\epsilon+a/2.
        \end{align*}
        Hence, $(x_1',x_\epsilon',x_2',y_1', y_\epsilon',y_2')\sim_1 (x_1,x_\epsilon,x_2,y_1, y_\epsilon,y_2)$, proving injectivity.
    \end{claimproof}

    \begin{claim}
        $\psi$ is well-defined and injective.
    \end{claim}

    \begin{claimproof}
        Choose any $(x_1,x_2,y_1,y_2)\in S$.
        We first observe that
        \begin{align*}
            x_2+(g(y_2)-\lambda_\epsilon)\bo &\in \Sigma_2, & y_1+h(x_1)\bo &\in T_1, \\ 
            (x_1-h(x_1)\bo,0) &\in \widetilde{\Sigma}_1, &(\lambda_\epsilon,y_2-(g(y_2)-\lambda_\epsilon)\bo) &\in \widetilde{T}_2.
        \end{align*}
        From this it is clear that
        \begin{equation*}
            \Big(x_1-h(x_1)\bo ~ , ~ 0 ~ , ~ x_2+(g(y_2)-\lambda_\epsilon)\bo ~ , ~  y_1+h(x_1)\bo ~ , ~ \lambda_\epsilon ~ , ~ y_2-(g(y_2)-\lambda_\epsilon)\bo \Big)\in \widetilde{S}.
        \end{equation*}
        For each $a,b\in\RR$ we have
        \begin{equation*}
            \begin{pmatrix}
                x_1+a\bo-h(x_1+a\bo)\bo\\
                0\\
                x_2+b\bo+(g(y_2-b\bo)-\lambda_\epsilon)\bo\\
                y_1-a\bo+h(x_1+a\bo)\bo\\
                \lambda_\epsilon\\
                y_2-b\bo-(g(y_2-b\bo)-\lambda_\epsilon)\bo
            \end{pmatrix}
            =
            \begin{pmatrix}
               x_1+(a-h(x_1)-a)\bo\\
               0\\
               x_2+(b+g(y_2)-b-\lambda_\epsilon)\bo\\
               y_1+(-a+h(x_1)+a)\bo\\
               \lambda_\epsilon\\
               y_2+(-b-g(y_2)+b+\lambda_\epsilon)\bo
            \end{pmatrix}
            =
            \begin{pmatrix}
                x_1-h(x_1)\bo\\
                0\\
                x_2+(g(y_2)-\lambda_\epsilon)\bo\\
                y_1+h(x_1)\bo\\
                \lambda_\epsilon\\
                y_2-(g(y_2)-\lambda_\epsilon)\bo
            \end{pmatrix}
        \end{equation*}
        proving that $\psi \colon S/\sim_2\longrightarrow \widetilde{S}/\sim_1$ is indeed well-defined.

        Now assume there exists $(x_1',x_2',y_1',y_2')\in S$ such that
        \begin{equation*}
            \begin{pmatrix}
                x_1-h(x_1)\bo\\
                0\\
                x_2+(g(y_2)-\lambda_\epsilon)\bo\\
                y_1+h(x_1)\bo\\
                \lambda_\epsilon\\
                y_2-(g(y_2)-\lambda_\epsilon)\bo
            \end{pmatrix}
            =
            \begin{pmatrix}
                x_1'-h(x_1')\bo+a\bo\\
                a\\
                x_2'+(g(y_2')-\lambda_\epsilon)\bo+a\bo\\
                y_1'+h(x_1')\bo-a\bo\\
                \lambda_\epsilon-a\\
                y_2'-(g(y_2')-\lambda_\epsilon)\bo-a\bo
            \end{pmatrix}
        \end{equation*}
        for some $a \in \mathbb{R}$.
        Direct comparison gives $a=0$.
        Using this with the substitutions $b=h(x_1')-h(x_1)$ and $c=g(y_2)-g(y_2')$, we get
        \begin{equation*}
            x_1'=x_1+b\bo , \qquad x_2'=x_2+c\bo, \qquad   y_1'=y_1-b\bo, \qquad
            y_2'=y_2-c\bo.
        \end{equation*}
        Hence, $(x_1',x_2',y_1',y_2') \sim_2 (x_1,x_2,y_1,y_2)$, proving injectivity.
    \end{claimproof}

    This now concludes the proof of the result.
\end{proof}

\end{document}